\documentclass{amsart}

\begin{document}
\newtheorem{thm}{Theorem}[section]
\newtheorem{pro}[thm]{Proposition}
\newtheorem{cor}[thm]{Corollary}
\newtheorem{lem}[thm]{Lemma}
\newtheorem*{problemHL}{Problem (HL)}
\theoremstyle{definition}
\newtheorem{dfn}[thm]{Definition}
\theoremstyle{remark}
\newtheorem{rmk}[thm]{Remark}
\newtheorem{exe}[thm]{Example}

\numberwithin{equation}{section}

\title[Extremal representations of functions of matrices]{Extremal representations of functions of matrices and applications to multivariate prediction}
\author{Michael Frank}
\address[M. Frank and A. Lasarow]{HTWK Leipzig, Fakult\"at f\"ur Informatik und Medien, Gustav-Freytag-Str.\,42a, D-04277 Leipzig, Germany}
\email{michael.frank@htwk-leipzig.de, andreas.lasarow@htwk-leipzig.de}
\author{Lutz Klotz}
\address[L. Klotz]{E.-v.-Brockdorff-Str. 16, D-04159 Leipzig, Germany}
\email{lutzklotz@t-online.de}
\author{Andreas Lasarow}

\begin{abstract}
Motivated by two seminal results of multivariate prediction theory by Helson and Lowdenslager and by Wiener and Masani we prove extremal representations of functions of matrices and derive their prediction-theoretic consequences. We also sketch a way to obtain matricial inequalities from our results. The main goal of the paper is the computation of the infimum of a set of values of the form ${\rm tr}(A \Delta A^*)$, where $\Delta$ is a given non-negative Hermitian $n \times n$ matrix and the choices for $A$ exhauste a certain set of $n \times n$ matrices. In particular, we focus on norm-bounded unit spheres with certain types of properties of unitary invariance, what allows an application of the theory of majorization.
\end{abstract}

\maketitle
\section{Introduction}
\label{s1}

Two equivalent major results in multivariate prediction theory by H.~Helson and D.~Lowdenslager  \cite{HL} and by N.~Wiener and P.~Masani \cite{WM} motivate new research efforts in matrix theory and multivariate prediction theory in the present paper. The goal is an extension of these results to a wider class of matrix sets to get a deeper insight into these aspects of multivariate prediction theory and to strengthen its ties with matrix theory. 

Generalizing the Szeg{\H{o}} infimum theorem, H.~Helson and D.~Lowdenslager proved the following remarkable result, cf.~\cite[Thm.~8]{HL}. The result was announced in an earlier paper \cite{Z} by V.~N.~Zasukhin, a younger colleague of A.~N.~Kolmogorov. The paper \cite{K} deals with an extension to certain $L^p$-spaces, and the paper \cite{Cheng} contains its further extension to the infinite-dimensional case. We denote the set of all complex-valued type $n \times n$ matrices of a certain fixed dimension $n$ by $\mathcal M$. 

\begin{thm} {\rm \cite[Thm.~8, Lemma 4 and 5]{HL}} \label{t1.1} \newline
    Let $M$ be a Borel measure on $[0,2 \pi)$ taking values in the convex cone  ${\mathcal M}^\geq$  of all non-negative $n \times n$ matrices of a fixed dimension $n$. Let ${\mathcal A}_1 = \{ A \in {\mathcal M}: {\rm det} (A)=1 \}$ and ${\mathcal P}_0$ the set of all finite sums of the form $\sum_{s \in {\mathbb N}} A_s e^{{\bf i}s ( \cdot )}$, $A_s \in {\mathcal M}$. Let $M$ have a Lebesgue decomposition ${\mathrm d} M  = W {\mathrm d}L + {\mathrm d} M_s$, where $W$ is a summable matrix function and $M_s$ is singular with respect to ${\mathrm d} L$, $L$ denoting the normalized Lebesgue measure on $[0,2 \pi)$. The infimum
    \begin{equation}  \label{f1.1}
         \inf \left\{ {\rm tr} \left( \int\limits_{0}^{2 \pi} (A-P) \, {\mathrm d}M \, (A-P)^* \right): A \in {\mathcal A}_1 , \, P \in {\mathcal P}_0 \right\}
    \end{equation}
can be expressed by $\, n \cdot {\rm exp}  \left( \frac{1}{n} \int_0^{2\pi} {\rm tr} ({\rm log} (W)) \, {\mathrm d} L \right) =  n \cdot {\rm exp} \left( \frac{1}{n} \int_{0}^{2 \pi}  {\rm log} ({\rm det} (W)) \, {\mathrm d} L \right)$. In particular, the infimum \eqref{f1.1} equals zero if and only if ${\rm log}({\rm det}(W))$ is not Lebesgue integrable.
\end{thm}

By the Lebesgue-Radon-Nikodym theorem $W$ can be described by a Radon-Nikodym derivative of the absolutely continuous part of $M$.
Considering $M$ as a (non-stochastic) spectral measure of an $n$-variate weakly stationary sequence ${\mathbf X}=(X(s) = (x_1(s), \ldots , x_n(s))^\top )_{s \in \mathbb Z}$ and taking into account a multivariate version of Kolmogorov's isomorphism, Theorem \ref{t1.1} can be given a probabilistic interpretation. Indeed, Theorem \ref{t1.1} is equivalent to the assertion that the distance of the set $\{ AX(0) : A \in {\mathcal A}_1 \}$ to the $\mathcal M$-linear space spanned by $\{ X(s): s \in -{\mathbb N} \}$ is equal to $n \cdot {\rm exp} \left( \frac{1}{n} \int_{0}^{2 \pi} {\rm log} ({\rm det} (W))\, {\mathrm d}  L \right)$, cf.~\cite[Thm.~11]{HL} for a precise formulation. The assertion of Theorem \ref{t1.1} is equivalent to Main Theorem I of another celebrated paper of that time, cf.~\cite [Subsection 7.10]{WM}. To state it denote the one-step-ahead prediction error matrix of $\mathbf X$ by $\Delta$, $\Delta = \inf \{ \int_{0}^{2 \pi} (I-P) \, {\mathrm d} M \, (I-P)^* : P \in {\mathcal P}_0 \}$, where $I$ denotes the identity matrix of $\mathcal M$, and the infimum is taken with respect to Loewner's semiordering. 

\begin{thm}   {\rm \cite [Subsection 7.10]{WM}}  \label{t1.2} \newline
   The determinant ${\rm det} (\Delta)$ is equal to ${\rm exp} \left(\int_{0}^{2 \pi} {\rm log} ( {\rm det} (W)  \, {\mathrm d} L\right)$ with the convention that ${\rm det}(\Delta)=0$ if ${\rm log}({\rm det}(W))$ is not Lebesgue integrable.
\end{thm}

Since the infimum \eqref{f1.1} is equal to 
\begin{eqnarray*}
\begin{split}
&\inf \left\{ {\rm tr} \left(A \int\limits_{0}^{2 \pi}  (I-A^{-1}P) \, {\mathrm d}M \, (I-A^{-1}P)^*A^*\right) : A \in {\mathcal A}_1, P \in {\mathcal P}_0 \right\} \\
& \quad \quad \quad \quad \quad =  \inf \left\{ {\rm tr} \left(A \int\limits_{0}^{2 \pi}  (I-P) \, {\mathrm d}M \, (I-P)^*A^*\right) : A \in {\mathcal A}_1, P \in {\mathcal P}_0 \right\} \\
& \quad \quad \quad \quad \quad = \inf \{ {\rm tr} (A \Delta A^*) : A \in {\mathcal A}_1 \} \\
& \quad \quad \quad \quad \quad = n \, {\rm det}(\Delta^{1/n}) \, ,
\end{split}
\end{eqnarray*}
cf.~\cite[Thm.~7.8.1]{HJ} for the last equality, the equivalence of Theorems \ref{t1.1} and \ref{t1.2} is obvious. 

Generalizing the problem of H.~Helson and D.~Lowdenslager one could ask for the infimum
\begin{equation}  \label{f1.2}
    \inf \left\{ {\rm tr} \left(\int\limits_{0}^{2 \pi}  (A-P) \, {\mathrm d}M \, (A-P)^* \right) : A \in {\mathcal A}, P \in {\mathcal P}({\mathcal O}) \right\} \, ,
\end{equation}
where $\mathcal A$ is a non-empty subset of $\mathcal M$, the symbol $\mathcal O$ denotes a non-empty subset of ${\mathbb Z} \setminus \{ 0 \}$, which is interpreted as observation set of the sequence ${\mathbf X}$, and ${\mathcal P}({\mathcal O})$ is the set of all $\mathcal M$-valued trigonometric polynomials of the form $\sum_{s \in  {\mathcal O}} A_s e^{{\bf i}s(\cdot)}$. In Section 6 we formulate this problem, which we will denote by Problem (HL), as a problem of prediction theory and call it ``multivariate prediction problem of Helson-Lowdenslager type''. If for given $\mathcal O$ some information on the corresponding prediction error matrix $\Delta ({\mathcal O}) = \inf \{ \int_{0}^{2 \pi} (I-P) \, {\mathrm d} M \, (I-P)^* : P \in {\mathcal P}({\mathcal O}) \}$ is known, one can try to determine the infimum \eqref{f1.2} following the proof of Theorem \ref{t1.1} by Theorem \ref{t1.2} as sketched above. In this case, a main goal will be the computation of the infimum 
\begin{equation}    \label{f1.3}
     \iota (\Delta ; {\mathcal A}) = \inf \{ {\rm tr} (A \Delta A^*): A \in {\mathcal A} \}
\end{equation}
for any $\Delta \in {\mathcal M}^\geq$. A large part of our paper deals with this problem. Although probably the computation of $\iota (\Delta ; {\mathcal A} )$ has little prediction-theoretical interest for many particular choices of the set $\mathcal A$, we think this kind of problem (and, of course, also the corresponding problem of determining $\sup \{ {\rm tr} (A \Delta A^*): A \in {\mathcal A} \}$) is of interest from the point of view of linear algebra and matrix theory. In particular, we are interested in cases where the infimum is attained. Moreover, as will be pointed out in Section 5, these problems deserve investigation since one could derive new inequalities from their solutions. In places we indicate open problems. 

In Section 2 some general and auxilary results are summarized. We also recall the simple variant of computing the infimum \eqref{f1.3} in case the set $\mathcal A$ is a hyperplane of $\mathcal M$. In Section 6 we demonstrate that its solution is of significant importance in multivariate prediction theory despite the simplicity of that particular case. 

Section 3 deals with the infimum of \eqref{f1.3} if $\mathcal A$ is a unitarily invariant subset of $\mathcal M$. In this situation the value $\iota (\Delta ; {\mathcal A})$ can be expressed by a permutation invariant function of the eigenvalues of $\Delta$, so the theory of majorization is helpful, cf.~\cite{MOA} for a comprehensive treatise of it.  For some observation sets $\mathcal O$ that are important in practice the matrix $\Delta({\mathcal O})$ cannot be calculated precisely. However, this matrix is at least known to be non-invertible for certain sets of spectral measures $M$. Consequently, we are searching for sets $\mathcal A$ such that $\iota(\Delta; {\mathcal A})$ can be described. In particular, sets $\mathcal A$ forcing $\iota(\Delta; {\mathcal A})=0$ are of special interest. 
What is more, Corollary \ref{c3.12} describes a large class of sets $\mathcal A$ such that for any $\Delta \in {\mathcal M}^\geq$ the infimum $\iota (\Delta ; {\mathcal A})$  is equal to the smallest eigenvalue of $\Delta$. 

In Sections 4 and 5 we deal with more concrete examples. Section 4 is mainly devoted to the case that $\mathcal A$ is a unit sphere of a unitarily invariant norm, but we also obtain a result if $\mathcal A$ is a unit sphere of a $C$-numerical radius. Section 5 contains further examples, particularly, a generalization of the case ${\mathcal A}={\mathcal A}_1$ is given, shown at the beginning of this introduction. Most results of Sections 4 and 5 provide examples of permutation invariant and concave, hence, Schur concave functions on $[0,+\infty)^n$. As indicated in Section 5, this observation opens the way to prove new inequalities. 

Section 6 shows applications to prediction theory of multivariate weakly stationary sequences. Beside the interpretation given by H.~Helson and D.~Lowdenslager (\cite[Thm.~11]{HL}) we give some further interpretations and indicate additional consequences. In particular, we obtain results concerning the multivariate prediction problem of Yu.~A.~Rozanov (\cite{R}), where the value $x_1(0)$ should be predicted not only on the basis of the realizations of $\mathbf X$ at the points of the observation set $\mathcal O$, but additionally, on the basis of the values $\{ x_j(0) : j \in \{ 2,\ldots,n\} \}$. 

The approach described in the present paper is applicable only if the prediction error matrix $\Delta({\mathcal O})$ itself or at least some information on it are known. So, we recall several investigations of practical interest from the literature where such information is available: 

\begin{itemize}
\item `classical' prediction: ${\mathcal O}= - {\mathbb N}$, cf.~Theorems \ref{t1.1} and \ref{t1.2} and the corresponding papers \cite{HL,WM} as well as the monograph \cite{R};
\item interpolation problem: ${\mathcal O} = (- {\mathbb N}) \cup {\mathbb N}$, cf.~\cite{MW};
\item Nakazi's problem: ${\mathcal O} = (- {\mathbb N}) \cup \{ 1,\ldots,k\}$ for some $k \in {\mathbb N}$, cf.~\cite{KL};
\item Pourahmadi's problem: ${\mathcal O} = (- {\mathbb N}) \setminus \{ -k \}$ for some $k \in \mathbb N$, cf.~\cite{KL};
\item periodic observations: ${\mathcal O} = \{ ks+m : s \in {\mathbb Z} \}$ for a certain $k \in {\mathbb N} \setminus \{ 1 \}$ and selected $m \in \{ 1,\ldots,(k-1) \}$, cf.~\cite{KM}.
\end{itemize}

\section{Preliminaries and the hyperplane case}
\label{s2}

In this section we fix denotations and recall some basic results used in the sequel. We consider $m \times n$ matrices as linear operators between complex finite-dimensio\-nal vector spaces, i.e.~from ${\mathbb C}^n$ to ${\mathbb C}^m$, where some orthonormal bases of these spaces are fixed to get a unique matrix representation of each linear operator. Vectors are imagined as column vectors, however, for the sake of a reasonable text reading, we write them as row vectors with a transposed sign at the end. 
For any $m \times n$ matrix $A$ with complex entries denote $A^{\top}$, $A^*$, $A^\dagger$, ${\mathcal R}(A)$, ${\mathcal N}(A)$, and ${\rm rk}(A)$ its transpose, adjoint, Moore--Penrose inverse, range, kernel, and rank, respectively. If $A$ is a quadratic matrix, then ${\rm tr}(A)$ and ${\rm det}(A)$ denote its trace and determinant values, respectively. If $A$ is a positive matrix $A^{1/2}$ denotes the unique positive matrix such that $(A^{1/2})^2=A$.
Throughout the present paper, let $n$ be a positive integer, where the symbol $\mathcal M$ stands for the algebra of complex $n \times n$ matrices, $\mathcal I$ for its open subset of all invertible matrices, $\mathcal U$ for its compact subset of all unitary matrices, and ${\mathcal M}^{\geq}$ for its convex cone of all non-negative Hermitian matrices.

For ${\mathcal A} \subseteq {\mathcal M}$ being a hyperplane, we want to resume the solution of the problem of computing the infimum of the value  ${\rm tr}(A \Delta A^*)$
\begin{equation} \label{f2.0}
    \iota(\Delta;{\mathcal A}) := \inf \left \{ {\rm tr}( A \Delta A^*) :  A\in {\mathcal A} \right \}
\end{equation} 
for any $\Delta\in {\mathcal M}^{\geq}$. 

In the following, the singular values of a matrix $A \in {\mathcal M}$ are denoted by $\sigma_j(A)$, where the dependence on $A$ will not be indicated in unambiguous cases, i.e. we write $\sigma_j$ in short instead of $\sigma_j(A)$. However, we emphasize that the singular values of $A$ are always assumed to be ordered decreasingly throughout this paper, i.e.
\begin{equation} \label{f2.1}
  \sigma_j \geq \sigma_{j+1} \, .
\end{equation}

Let $\|A\|_2$ be the Schatten $2$-norm (or Euclidean or Frobenius norm) of $A$, what
implies
\[ 
    \|A\|_2 ^2= {\rm tr}(A^* A) = \sum_{k,l} |a_{k l}|^2 = \sum_{j} \sigma_j^2 \, ,  
\]
where $A=(a_{k l})$ is a quadratic matrix with complex entries $a_{k l}$, $A \in {\mathcal M}$ of arbitrary size $n$. The symbol $\Delta$ always stands for a matrix of ${\mathcal M}^{\geq}$.  Let $\lambda_1,\ldots,\lambda_n$ be the eigenvalues of $\Delta$, where
\[ 
    \lambda_j \geq \lambda_{j+1}, \quad j\in\{1,\ldots,n-1\},  
\]
in accordance with \eqref{f2.1}. Moreover, let ${\mathcal U}_{\Delta}$ denote the set of all $U\in {\mathcal U}$ such that
\begin{equation} \label{f2.3}
  \Delta = U \, {\rm diag}(\lambda_n,\ldots,\lambda_1) \, U^*,
\end{equation}
where ${\rm diag}(\lambda_n,\ldots,\lambda_1)$ stands abbreviated for the diagonal matrix in $\mathcal M$ with entries on the principal diagonal $\lambda_n,\ldots,\lambda_1$ incipient top left. 
Note that, therefore, the eigenvalues on the principal diagonal at the right-hand side of \eqref{f2.3} are ordered increasingly, which simplifies our presentation slightly. We denote the identity matrix by $I$.

Recall that a real-valued function $f$ on a convex cone $\mathcal K$ of an $\mathbb R$-linear space is called positively homogeneous, concave, and superadditive if
\[ 
   f(ax) = af(x),\,\,\,
   f\bigl(bx+(1-b)y\bigr) \geq bf(x)+(1-b)f(y),\,\,\, \mbox{and}\,\,\,
   f(x+y\bigr) \geq f(x)+f(y), 
\]
respectively, for all $a\in (0,+\infty)$, $b\in[0,1]$, $x,y\in {\mathcal K}$.
Moreover, if $0\in {\mathcal K}$, we suppose that $f(0)=0$. Note that, here and henceforth, the notation $0$ is used for the null element in the corresponding space.

The set of all non-negative and positively homogeneous functions on $\mathcal M$, which are different from the zero function, is denoted by $\mathcal F$. The proofs of the following facts are omitted.

\begin{lem} \label{l2.1}
    Let $f$ be a function in the set $\mathcal F$ of all non-negative and positively homogeneous functions on $\mathcal M$ different from the zero function.
\begin{itemize}
        \item[(i)]  The function $f^p$ is not concave if $p\in (1,+\infty)$ and not superadditive if $p\in (0,1)$.
        \item[(ii)]  The function $f$ is concave if and only if it is superadditive. In this case,  $f^p$ is concave
                                for $p\in (0,1]$ and superadditive for $p\in [1,+\infty)$.
\end{itemize}
\end{lem}

In the following, $\mathcal A$ always denotes a non-empty subset of $\mathcal M$, and we use the notation given in \eqref{f2.1} based on a matrix $\Delta$. If the set $\mathcal A$ is characterized by a certain family of conditions $\mathcal C$ we shall write $\iota(\Delta;{\mathcal C})$ for short. For example, if ${\mathcal A}=\{A\in {\mathcal M} : {\rm tr}(A) = 1\}$, we write $\iota(\Delta; {\rm tr}(A) = 1)$ instead of $\iota(\Delta;\{A\in {\mathcal M} : {\rm tr}(A) = 1\})$. If the infimum stated in \eqref{f2.0} is attained, we call it a minimum and write the lowercase letter $m$ in place of $\iota$. Moreover, the set of all $A\in {\mathcal A}$, where the minimum is attained, will be denoted by ${\mathcal A}(\Delta)$. In particular, ${\mathcal A}(\Delta)=\emptyset$ if and only if the infimum is not attained.

Note that the function $\iota(\,\cdot\,; {\mathcal A})$ on the convex cone  ${\mathcal M}^{\geq}$ is non-negative, positively homogeneous, concave, and superadditive, hence, increasing with respect to Loewner's semiordering.
From concavity follows that $\iota(\,\cdot\,; {\mathcal A})$ is continuous at the set ${\mathcal M}^{\geq} \cap {\mathcal I}$ of inner points of ${\mathcal M}^{\geq}$. As an infimum of continuous functions, it is upper semicontinuous on ${\mathcal M}^{\geq}$.
If ${\mathcal A}$ is a bounded set, there exists a constant $c\in(0,+\infty)$ such that ${\rm tr}( A \Delta A^*)\leq c \cdot {\rm tr}(\Delta)$ for all $A\in {\mathcal A}$, hence, $\iota(\,\cdot\,;{\mathcal A})$ is continuous.  In general, $\iota(\,\cdot\,;{\mathcal A})$ is not necessarily continuous at boundary points of ${\mathcal M}^{\geq}$.

\begin{exe}
 If $n=2$ and
\[  
   {\mathcal A} = \left\{
        \left(\begin{array}{cc}
               1 & -k \\
              -k & k^2
              \end{array}
       \right) : k\in {\mathbb N} \right\} \, , \quad
    \Delta = 
       \left(\begin{array}{cc}
               1 & 0 \\   
               0 & 0
             \end{array}
       \right) \, , 
\]
then $m(\Delta;{\mathcal A}) = 2$. But, using the notation
\[  
    \Delta_k = 
      \left(\begin{array}{cc}
             1 & k^{-1} \\
             k^{-1} & k^{-2}
            \end{array}
      \right) \, ,   
 \]
the sequence $(\Delta_k)_{k\in {\mathbb N}}$ tends to $\Delta$ if $k\to+\infty$ and $m(\Delta_k;{\mathcal A}) = 0$
for all $k\in {\mathbb N}$.
\end{exe}

For future reference the following lemma summarizes further properties of $\iota(\Delta;{\mathcal A})$ for non-empty subsets $\mathcal A$ of $\mathcal M$, whose easy verifications are omitted.

\begin{lem} \label{l2.2}
Let $\mathcal A$ and ${\mathcal A}'$ be non-empty subsets of $\mathcal M$ and $\Delta \in {\mathcal M}^\geq$.
   \begin{itemize}
       \item[(i)]  Suppose that $\iota(\Delta;{\mathcal A}') \leq \iota(\Delta;{\mathcal A})$ and  ${\rm tr}( A_k \Delta A_k^*) \leq \iota(\Delta;{\mathcal A}') + k^{-1}$ for a sequence $(A_k)_{k\in {\mathbb N}}$ of matrices of $\mathcal A$. Then $\iota(\Delta;{\mathcal A}) = \iota(\Delta;{\mathcal A}')$.
       \item[(ii)] For any norm-dense subset ${\mathcal A}'$ of ${\mathcal A}$, the functions $\iota(\,\cdot\,;{\mathcal A})$ and $\iota(\,\cdot\,;{\mathcal A}')$ coincide.
       \item[(iii)] If the minimum $m(\Delta;{\mathcal A})$ is realized at $A\in {\mathcal A}$, then $m(\Delta;{\mathcal A}) = m(\Delta;{\mathcal A}')$ for all ${\mathcal A}'$ with $\{A\} \subseteq {\mathcal A}' \subseteq {\mathcal A}$.
       \item[(iv)] If ${\mathcal A}$ is such that $A\in {\mathcal A}$ yields $(A^* A)^{1/2} \in {\mathcal A}$, then the functions $\iota(\,\cdot\,;{\mathcal A})$ and $\iota(\,\cdot\,;{\mathcal A}')$ coincide for all ${\mathcal A}'$ with  ${\mathcal A} \cap  {\mathcal M}^{\geq} \subseteq {\mathcal A}' \subseteq {\mathcal A}$.
       \item[(v)] If $\{ A_\tau : \tau\in {\mathcal T}\}$ is a family of non-empty subsets of $\mathcal M$, then the function
 $\iota(\,\cdot\,;\bigcup_{\tau \in {\mathcal T}}\!\!A_\tau)$ is given by $\inf \{  \iota(\,\cdot\,;{\mathcal A}_\tau) : \tau\in {\mathcal T}\}$.
\item[(vi)] If ${\mathcal U}_1$ is a non-empty subset of $\mathcal U$, then $\iota(\,\cdot\,;{\mathcal A}) = \iota(\,\cdot\,;\bigcup_{U \in {\mathcal U}_1}\!U{\mathcal A})$.
\end{itemize}
\end{lem}

The next lemma reveals that for many sets $\mathcal A$ it is sufficient to deal with invertible matrices $\Delta$, i.e.~with matrices $\Delta$ such that ${\mathcal R}(\Delta)$ is the entire space.

\begin{lem} \label{l2.3}
    For any ${\mathcal A} \subset {\mathcal M}$ and $\Delta \in {\mathcal M}^\geq$, let $P$ be the orthoprojection onto ${\mathcal R}(\Delta)$. The equality $\iota(\Delta;{\mathcal A}) = 0$ is true if and only if there exists a sequence $(A_k)_{k\in{\mathbb N}}$ of matrices of $\mathcal A$ with $\lim_{k\to+\infty} A_kP = 0$. In this case, ${\mathcal A}(\Delta) = \{ A\in{\mathcal A} : {\mathcal R}(\Delta) \subseteq {\mathcal N}(A) \}$.
\end{lem}

\begin{proof}
Since the equality $\lim_{k\to+\infty} {\rm tr}(A_k\Delta A_k^*) = 0$ is equivalent to the equality $\lim_{k\to+\infty} A_kP\Delta^{1/2} = 0$, which in turn is equivalent to $\lim_{k\to+\infty} A_kP = 0$, the first assertion follows. The second one is obvious.
\end{proof}

The second assertion of the preceding lemma has a partial extension to invertible matrices $\Delta$.

\begin{pro} \label{p2.4}
  Let ${\mathcal A} \subset {\mathcal M}$. For all $\Delta \in{\mathcal M}^\geq$, if for all $A\in{\mathcal A}$
     \begin{equation} \label{f2.5}
         \|A\|_2 \geq 1 
     \end{equation}
  then    
     \begin{equation} \label{f2.4}
         \iota(\Delta;{\mathcal A}) \geq \lambda_n \, .
     \end{equation}
  If the condition \eqref{f2.5} is not valid then there exists a $\Delta$ (e.g. $\Delta = I$) such that \eqref{f2.4}  is not valid for that $\Delta$.
  Furthermore, if \eqref{f2.5} is true and $\Delta \in {\mathcal I}$, then the equality $m(\Delta;{\mathcal A}) = \lambda_n$ is fulfilled if and only if the set ${\mathcal A}' = \{ A\in {\mathcal A} : {\mathcal R}(\Delta-\lambda_nI) \subseteq {\mathcal N}(A) \mbox{ and } \|A\|_2=1 \}$ is not empty. In this case ${\mathcal A}(\Delta) = {\mathcal A}'$.
\end{pro}

\begin{proof}
From ${\rm tr}\bigl( A(\Delta-\lambda_nI)A^* \bigr) \geq 0$ follows
\begin{equation} \label{f2.6}
     {\rm tr}(A\Delta A^*) \geq   \lambda_n \|A\|_2^2 \, .
\end{equation}
Thus, \eqref{f2.5} yields \eqref{f2.4}. Conversely, if $\|A\|_2<1$ for some $A\in {\mathcal A}$ and $\Delta = I$ we have
\[ 
    \iota(\Delta;{\mathcal A})  = \iota(I;{\mathcal A}) \leq {\rm tr}(A A^*) =  \|A\|_2^2 < 1 = \lambda_n (I) \, . 
\]
Now, assume the inequality  \eqref{f2.5} to be satisfied and $\Delta \in {\mathcal I}$ be such that $m(\Delta;{\mathcal A}) = \lambda_n$. Relation \eqref{f2.6} implies that
\begin{equation} \label{f2.7}
    \|A\|_2=1 \quad\mbox{and} \quad  {\rm tr}\bigl(A(\Delta-\lambda_nI)A^*\bigr) = 0 \, ,
\end{equation}
for any $A \in {\mathcal A}(\Delta)$, and hence $A\in {\mathcal A}'$. Conversely, if $A\in {\mathcal A}'$, then \eqref{f2.7} is true, which gives
${\rm tr}(A\Delta A^*) = \lambda_n$. Thus, $m(\Delta;{\mathcal A}) = \lambda_n$ follows, and $A\in{\mathcal A}(\Delta)$.
\end{proof}

Now, we resort to a first concrete setting. If the set $\mathcal A$ is a hyperplane, the behaviour of $\iota(\,\cdot\,;{\mathcal A})$ is known and easily to describe, cf.~\cite[Problem 7.4.12]{HJ_alt} and \cite[2.6.P4]{HJ} for a related result. Despite the simplicity, the hyperplane case is of considerable importance in the prediction theory. For convenience of the reader we give a full proof and add some details and examples, which will be applied to multivariate prediction problems in Section \ref{s6}.

\begin{pro} \label{p2.5}
    Let $X \in {\mathcal M} \setminus \{ 0 \}$. Consider the sets ${\mathcal A}_X:= \{ A\in {\mathcal M} : {\rm tr}(A X^*) = 1 \}$ and ${{\mathcal A}_X':= \{ A\in {\mathcal M} : |{\rm tr}(A X^*)| = 1 \}}$. 
\begin{itemize}
   \item[(i)]
 If ${\mathcal N}(\Delta) \subseteq {\mathcal N}(X)$, then  $m(\Delta;{\mathcal A}_X) = m(\Delta;{\mathcal A}_X') = \bigl({\rm tr}(X\Delta^{\dagger}X^*) \bigr)^{-1}$ and ${\mathcal A}_X(\Delta) = \{ A\in {\mathcal M}: AP = \bigl({\rm tr}(X\Delta^{\dagger}X^*) \bigr)^{-1}X\Delta^{\dagger}\}$,
${\mathcal A}_X'(\Delta) = \{ A\in {\mathcal M} : AP = \exp({\bf i} a) \cdot \bigl({\rm tr}(X\Delta^{\dagger}X^*) \bigr)^{-1}X\Delta^{\dagger} \mbox{ for some } a\in [0,2\pi) \}$, where $P$ is the orthoprojection onto~${\mathcal R}(\Delta)$.
   \item[(ii)]
If ${\mathcal N}(\Delta)$ is not a subset of ${\mathcal N}(X)$, then  $m(\Delta;{\mathcal A}_X) = m(\Delta;{\mathcal A}_X') = 0$.
\end{itemize}
\end{pro}

\begin{proof}
(i): Condition ${\mathcal N}(\Delta) \subseteq {\mathcal N}(X)$ is equivalent to $PX^*=X^*$. Thus, if ${\mathcal N}(\Delta) \subseteq {\mathcal N}(X)$, then
\[
    1 = |{\rm tr}(AX^*) |^2   =    |{\rm tr}(A\Delta^{1/2}(\Delta^{\dagger})^{1/2}X^*) |^2
    \leq {\rm tr}(A\Delta A^*) \cdot  {\rm tr}(X\Delta^{\dagger} X^*) 
\]
with equality if and only if $A\Delta^{1/2} = c \cdot X(\Delta^{\dagger})^{1/2}$ or, equivalently, $AP = c \cdot X\Delta^{\dagger}$ for some $c \in \mathbb C$. Clearly, one can derive $c = \bigl({\rm tr}(X\Delta^{\dagger}X^*) \bigr)^{-1}$ if $A\in {\mathcal A}_X$ and $|c| = \bigl({\rm tr}(X\Delta^{\dagger}X^*) \bigr)^{-1}$  if $A\in {\mathcal A}_X'$, which completes the proof of part (i).

(ii): Now, let ${\mathcal N}(\Delta)$ be not a subset of ${\mathcal N}(X)$. Then the matrix $(I-P)X^*$ is not the zero matrix. Hence, there exists a $C\in \mathcal M$ with ${\rm tr}\bigl(C(I-P)X^* \bigr) = 1$. Setting $A:=C(I-P)$, we have $A\in {\mathcal A}_X\subseteq {\mathcal A}_X'$ and $AP$ is the zero matrix, which yields $m(\Delta;{\mathcal A}_X) = 0$ and $m(\Delta;{\mathcal A}_X') = 0$ by Lemma \ref{l2.3}.
\end{proof}

As an immediate conclusion we get the following.

\begin{cor} \label{c2.6}
    Let $X \in {\mathcal M} \setminus \{ 0 \}$. Consider the sets ${\mathcal A}_X:= \{ A\in {\mathcal M} : {\rm tr}(A X^*) = 1 \}$ and ${{\mathcal A}_X':= \{ A\in {\mathcal M} : |{\rm tr}(A X^*)| = 1 \}}$. 
    The equality $m(\Delta;{\mathcal A}_X) = m(\Delta;{\mathcal A}_X')$ is fulfilled for all $\Delta$.
\end{cor}

Let $r:={\rm rk}(X)$. According to the orthogonal decomposition ${\mathbb C}^n = {\mathcal R}(X^*) \oplus {\mathcal N}(X)$ the matrix has a block representation $X = (X_1 \,\, 0)$ with an $n\times r$ matrix $X_1$ of rank $r$. If $r < n$, let
\[ 
   \Delta = \left(\begin{array}{cc}
               \Delta_{11} & \Delta_{12} \\
               \Delta_{12}^{*} & \Delta_{22}
   \end{array}\right)  
\]
be the corresponding partition of $\Delta$ with $r\times r$ matrix $\Delta_{11}$ and let
$\Delta/\Delta_{22}:=\Delta_{11} - \Delta_{12}\Delta_{22}^{\dagger}\Delta_{12}^*$ be its generalized Schur complement. If $r = n$, for harmonization, let $\Delta/\Delta_{22}:=\Delta$. Some consequences of Proposition \ref{p2.5} can be stated in terms of $\Delta/\Delta_{22}$.

\begin{cor} \label{c2.7}
    Let $X \in {\mathcal M} \setminus \{ 0 \}$ and ${\mathcal A}_X:= \{ A\in {\mathcal M} : {\rm tr}(A X^*) = 1 \}$.
\begin{itemize}
   \item[(i)] The minimum $m(\Delta;{\mathcal A}_X)$ is positive if and only if $\Delta/\Delta_{22}$ is invertible.
   \item[(ii)] If $\Delta/\Delta_{22}$ is invertible, then $m(\Delta;{\mathcal A}_X) = \bigl({\rm tr}(X_1 (\Delta/\Delta_{22})^{-1} X_1^*) \bigr)^{-1}$.
   \item[(iii)] Let $\Delta_{22}$ be invertible. Then the minimum $m(\Delta;{\mathcal A}_X)$ is positive if and only if $\Delta$ is invertible.
   \item[(iv)] If $r=1$ and $\Delta_{22}$ invertible, then
    \begin{equation} \label{f2.8}
         m(\Delta;{\mathcal A}_X) =  \|X_1\|_2^{-2} \, \det (\Delta_{22}^{-1}) \, {\rm det}(\Delta) .
\end{equation}
\end{itemize}
\end{cor}

\begin{proof}
(i): Recall that $\Delta/\Delta_{22}$ is invertible if and only if ${\mathcal N}(\Delta) \subseteq {\mathcal N}(X)$, and apply Proposition~\ref{p2.5}.

(ii): Let $\Delta/\Delta_{22}$ be invertible. Then we have $(\Delta^{\dagger})_{11} = (\Delta/\Delta_{22})^{-1}$. Hence, by using (i) and Proposition~\ref{p2.5}(i) we get
\[ 
    m(\Delta;{\mathcal A}_X) = \bigl({\rm tr}(X\Delta^{\dagger}X^*) \bigr)^{-1} = \bigl({\rm tr}(X_1 (\Delta/\Delta_{22})^{-1} X_1^*)     \bigr)^{-1} . 
\]

(iii): If $\Delta_{22}$ is invertible, then the matrix $\Delta$ is invertible if and only if
 $\Delta/\Delta_{22}$ is invertible. Thus, (i) implies (iii).

(iv): Let $\Delta_{22}$ be invertible. If $\Delta$ is not invertible, then $\Delta/\Delta_{22}$ is not invertible or, equivalently, ${\mathcal N}(\Delta)$ is not a subset of ${\mathcal N}(X)$. By Proposition~\ref{p2.5}(ii) $m(\Delta;{\mathcal A}_X)=0$ and, therefore, \eqref{f2.8} follows. Now, let $\Delta$ be invertible.  Since $r=1$, we have $X_1 \in {\mathbb C}^n$.  Hence, (ii) implies  $m(\Delta;{\mathcal A}_X) = \|X_1\|_2^{-2} (\delta_{11}^{(\dagger)})^{-1}$, where $\delta_{11}^{(\dagger)}$ stands for the $(1,1)$-element of $\Delta^\dagger$.  Because of $\delta_{11}^{(\dagger)} = (\det \Delta)^{-1} \cdot \det (\Delta_{22})$, we get equality \eqref{f2.8}.
\end{proof}

\begin{exe} \label{e2.8}
 Denote the $(j,k)$-element of $A$ by $a_{jk}$.
 
(i) Choosing $X={\rm diag}(1,\ldots,1,0,\ldots,0)$ with exactly $r$ numbers $1$ on the principal diagonal, we obtain
\[ 
    m\left(\Delta; \sum_{k=1}^r a_{kk} = 1 \right) =  m\left(\Delta; \left|\sum_{k=1}^r a_{kk}\right| = 1 \right) \,  , 
\]
 where $m\Bigl(\Delta; \sum_{k=1}^r a_{kk} = 1 \Bigr) = \bigl({\rm tr}(\Delta/\Delta_{22})^{-1} \bigr)^{-1}$ if ${\mathcal R}(\Delta)$ contains the first $r$ vectors $e_1,\ldots,e_r$ of the canonical basis of ${\mathbb C}^n$, and $m\Bigl(\Delta; \sum_{k=1}^r a_{kk} = 1 \Bigr)  = 0$ if not, cf. \cite[Theorem~5.12]{A}. In particular, we get $m(\Delta; {\rm tr}(A )= 1) = m(\Delta; |{\rm tr}(A)| = 1)$, where $m(\Delta; {\rm tr}(A) = 1) = ({\rm tr}(\Delta^{-1}))^{-1}$ if $\Delta\in \mathcal I$, and $m(\Delta; {\rm tr}(A) = 1) = 0$ otherwise.

(ii) Set $s^{(r)}:=\sum_{k=1}^r e_k = (1,\ldots,1,0,\ldots,0)^\top$ with exactly $r$ positions unequal zero. If  
\[
    X=\left(\begin{array}{c} (s^{(r)})^{\top} \\ 0 \end{array}\right)
\]
we get
\[ 
   m\left(\Delta; \sum_{k=1}^r a_{1k} = 1 \right) =  m\left(\Delta; \left|\sum_{k=1}^r a_{1k}\right| = 1 \right)   \, , 
\]
where $m\Bigl(\Delta; \sum_{k=1}^r a_{1k} = 1 \Bigr)  = \bigl(  \sum_{j,k=1}^r \delta_{jk}^{(\dagger)}  \bigr)^{-1}$ for any $\Delta$ with $s^{(r)}\in {\mathcal R}(\Delta)$, and  $m\Bigl(\Delta; \sum_{k=1}^r a_{1k} = 1 \Bigr)  = 0$ if not.
 
(iii) Let  $X=\left(\begin{array}{cc} s^{(r)} & 0 \end{array}\right)$. From (i) and Corollary~\ref{c2.7}(iv) we derive
\[ 
    m\left(\Delta; \sum_{j=1}^r a_{j1} = 1 \right)  =  m\left(\Delta; \left|\sum_{j=1}^r a_{j1}\right| = 1 \right)  =  r^{-1} \cdot m\left(\Delta; a_{11} = 1 \right) \, .  
\]
\end{exe}

We also note the following consequence of Proposition \ref{p2.5} and Corollary~\ref{c2.7}.

\begin{cor} \label{c2.9}
     Let $X\in {\mathcal M}\setminus\{0\}$ and the function $f_X$ be defined on ${\mathcal M}^{\geq}$ as follows: $f_X(\Delta) = \bigl({\rm tr}(X\Delta^{\dagger}X^*) \bigr)^{-1}=({\rm tr}(X(\Delta / \Delta_{22})^{-1} X^*))^{-1}$ if ${\mathcal N}(\Delta) \subseteq {\mathcal N}(X)$, and $f_X(\Delta) = 0$ else. Then the function $f_X$ is concave and superadditive.
\end{cor}

One can also derive a well-known generalization of Bergstr\"om's inequality, cf. \cite[page~65]{Liu}. For this, let $\Delta_k$ denote the right lower $(n-k) \times (n-k)$ corner of $\Delta$.

\begin{cor} \label{c2.10}
      For all $k\in\{1,\ldots,n-1\}$, the function $f_k$ defined on ${\mathcal M}^{\geq} \cap {\mathcal I}$  via precept $f_k(\Delta):= \bigl(({\rm det}(\Delta_k))^{-1} \, {\rm det}(\Delta)\bigr)^{1/k})$ is  concave and superadditive.
\end{cor}

\begin{proof}
Corollary~\ref{c2.7}(iv) reveals that $f_1$ is concave and superadditive. Similarly, any function $\Delta_k \to ({\rm det}(\Delta_{k+1}))^{-1}  {\rm det}(\Delta_k)$ is concave and superadditive. Since monotonicity, concavity, and superadditivity of the geometric mean imply that the $k$-th root of the product of $k$ non-negative concave or superadditive functions is concave or superadditive, respectively, the assertion follows.
\end{proof}

\section{Unitary invariance}
\label{s3}

The most valuable results on $\iota(\,\cdot\,;{\mathcal A}) = \inf \{ {\rm tr}(A \Delta A^*) : A \in {\mathcal A} \}$ can be established if ${\mathcal A}$ has various properties of unitary invariance. A non-empty subset $\mathcal A$ of $\mathcal M$ is called right unitarily invariant, unitary similarity invariant or unitarily invariant if ${\mathcal A}U = {\mathcal A}$, $U^*{\mathcal A }U = {\mathcal A}$ or $V{\mathcal A}U={\mathcal A}$, respectively, for all $U,V \in {\mathcal U}$. If $\mathcal A$ is right unitarily invariant or unitary similarity invariant then it is called weakly unitarily invariant. Note, that our choice of notions differs from the notion of weakly unitarily invariant (wui) norms in  \cite[Section IV.4]{Bh}. Analogously to the denotation for diagonal matrices, we denote a block-diagonal matrix with a sequence of quadratic blocks $(A_j)$, $j=1,\ldots,n$, of possibly different sizes as ${\rm diag}(A_1, \ldots , A_n)$.

\begin{pro} \label{p3.1}
    Let ${\mathcal A} \subseteq \mathcal M$ be weakly unitarily invariant, Then the function $\iota(\,\cdot\,;{\mathcal A})$ is continuous on ${\mathcal M}^\geq$.
\end{pro}

\begin{proof}
Assume that $\iota(\,\cdot\,;{\mathcal A})$ is not continuous at point $\Delta$. Then $r={\rm rk}(\Delta)$ is smaller than $n$. Since the function $\iota(\,\cdot\,;{\mathcal A})$ is upper-semi-continuous, there exist $\varepsilon > 0$ and sequences $(\Delta_k)$ with $\Delta_k \in {\mathcal M}^\geq$, $(A_k)$ with $A_k \in \mathcal A$, such that $\lim_{k \to +\infty}\Delta_k = \Delta$ and
\begin{equation} \label{f3.1}
      {\rm tr} (A_k\Delta_kA_k^*) \leq \iota(\,\Delta\,;{\mathcal A}) - \varepsilon
\end{equation}
for any $k \in \mathbb N$. Let $a = \iota(\,\Delta\,;{\mathcal A})$ and $U_k \in {\mathcal U}_{\Delta_k}$, $\Delta_k = U_k \, {\rm diag}(\lambda_{n k}, \ldots , \lambda_{1 k}) \, U_k^*$. By the compactness of $\mathcal U$ we can suppose that $U=\lim_{k \to +\infty}U_k$ exists, what yields $\lim_{k \to +\infty} {\rm diag}(\lambda_{n k}, \ldots , \lambda_{1 k}) = \lim_{k \to +\infty} U_k^* \Delta_k U_k = U^* \Delta U = {\rm diag}(\lambda_n, \ldots ,\lambda_1)$. Since ${\rm tr}(U_k^*A_kU_kU_k^*\Delta_kU_kU_k^*A_k^*U_k) \leq a-\varepsilon$ and the assumption of weak unitary invariance of $\mathcal A$ implies $U_k^*A_kU_k \in \mathcal A$ or $A_kU_k \in \mathcal A$, we can assume from the very beginning of the present proof that $\Delta = {\rm diag}(0, \Delta')$ with an invertible $r \times r$ matrix $\Delta'$ and $\Delta_k={\rm diag}(\Delta_k'' , \Delta_k')$ with $\lim_{k \to +\infty} \Delta_k'' = 0$ and $\lim_{k \to +\infty} \Delta_k' = \Delta'$. So, partition $A$ as $A=(A'',A')$, where $A'$ is an $n \times r$ matrix, and set ${\mathcal A}'=\{ A' : A \in {\mathcal A} \}$. The equality ${\rm tr}(A \Delta A^*) = {\rm tr}(A' \Delta' (A')^*)$ yields $\iota(\,\Delta'\,;{\mathcal A}')=a$. Since the function $\iota(\,\cdot\,;{\mathcal A}')$ is continuous on the set of positive Hermitian $r \times r$ matrices, the inequality ${\rm tr}(A_k\Delta_k A_k^*)= {\rm tr}(A_k''\Delta_k''(A_k'')^*) + {\rm tr} (A_k'\Delta_k'(A_k')^*) > a-\varepsilon$ is valid for sufficiently large $k$, a contradiction to \eqref{f3.1}
\end{proof}

The next theorem is partially already known. We define a set $\mathcal D$ by
\begin{equation} \label{f3.2}
    {\mathcal D} = \{ (x_1, \ldots ,x_n)^\top \in [0,+\infty)^n : x_1 \geq \ldots \geq x_n \}
\end{equation}
and its subset ${\mathcal D}_+$ by  ${\mathcal D}_+ = {\mathcal D} \cap (0,+\infty)^n$. Note, that according to our convention \eqref{f2.1} the vector consisting of the singular values of a matrix $A \in \mathcal M$ belongs to $\mathcal D$. 

\begin{thm} \label{t3.2}
    Let $(\sigma_1^{(0)}, \ldots , \sigma_n^{(0)}) \in {\mathcal D}$ and ${\mathcal A}_\sigma = \{ A \in {\mathcal M} : \sigma_j(A)=\sigma_j^{(0)}, \, j \in \{ 1,\ldots,n\} \}$. For any $\Delta$ we have
\begin{equation} \label{f3.3}
    m(\Delta;{\mathcal A}_\sigma) = \sum_{j=1}^n \lambda_{n-j+1}(\sigma_j^{(0)})^2
\end{equation}
and
\begin{equation} \label{f3.4}
    {\mathcal A}_\sigma (\Delta) = \{  V \, {\rm diag}(\sigma_1^{(0)}, \ldots , \sigma_n^{(0)}) \, U^* : V \in {\mathcal U}, U \in {\mathcal U}_\Delta \} \, .
\end{equation}
\end{thm}

The relation \eqref{f3.3} is a classical result and the main assertion of this theorem, cf.~\cite[Thm.~20.A.2]{MOA}. To prove \eqref{f3.4} we use an auxiliary result:

\begin{lem} \label{l3.3}
   Let $(\lambda_1^{(0)}, \ldots ,\lambda_n^{(0)})^\top, \, ({\mu_n}, \ldots ,{\mu_1})^\top \in{\mathcal D}$, $\mu_{\ell+1} \geq \mu_\ell$ for $\ell = 1, \ldots , n-1$. Then 
\begin{equation} \label{f3.5}
       \min  \left\{ {\rm tr}(B \, {\rm diag}(\mu_1,\ldots,\mu_n)) : B \in {\mathcal M}^\geq, \lambda_j(B)=\lambda_j^{(0)} , j \in \{ 1,\ldots,n\}\right\} = \sum_{j=1}^n \mu_j\lambda_j^{(0)} \, .
\end{equation}   
Let $k \in \{1, \ldots , n-1\}$ and $1=n_0 <n_1<\ldots<n_k=n+1$ be such that $\mu_{n_{s-1}} = \ldots = \mu_{n_s - 1} < \mu_{n_s}$, $s \in \{1,\ldots,k\}$. The minimum \eqref{f3.5} is attained if and only if $B$ has the block-diagonal form 
\begin{eqnarray} \label{f3.6}
   B & = & {\rm diag}(U_0 \, {\rm diag}(\lambda_1^{(0)}, \ldots ,\lambda_{n_1-1}^{(0)}) \, U_0^*, U_1 \, {\rm diag}(\lambda_{n_1}^{(0)}, \ldots , \lambda_{n_2-1}^{(0)}) \, U_1^*,  \ldots \\ 
      &    & \ldots , U_{k-1} \, {\rm diag}(\lambda_{n_{k-1}}^{(0)}, \ldots ,\lambda_{n_k-1}^{(0)}) \, U_{k-1}^*) \, , \nonumber
\end{eqnarray}
where $U_s$ is an $(n_{s+1}-n_s) \times (n_{s+1}-n_s)$ unitary matrix for $s \in \{ 0,\ldots,k-1\}$. 
\end{lem}

\begin{proof}
The relation \eqref{f3.5} is an immediate consequence of \eqref{f3.3}. The sufficiency of \ref{f3.6} is obvious. Its necessity can be proved by induction on the number $k$ of different eigenvalues of $B$. If $k=1$ the matrix $B$ is a scalar multiple of the identity matrix, and the assertion is trivial. For arbitrary $k \in \{1,\ldots,n-1\}$ write $B=(B_{st})_{s,t \in \{ 0,\ldots,k-1\}}$, where $B_{st}$ is an $(n_{s+1}-n_s) \times (n_{t+1}-n_t)$ matrix. According to a well-known inequality, cf.~\cite[Prop.~6.A.3]{MOA}, the value ${\rm tr}(B \, {\rm diag}(\mu_1, \ldots , \mu_n)) = \sum_{s=0}^{k-1} \mu_{n_s} {\rm tr}(B_{ss})$ admits its minimum if ${\rm tr}(B_{00})$ is maximal and $\sum_{s=1}^{k-1} \mu_{n_s} {\rm tr}(B_{ss})$ is as small as possible. By the interlacing theorem of eigenvalues of a Hermitian matrix, we get that  ${\rm tr}(B_{00})$ is maximal if and only if $B$ has the form $B={\rm diag}(U_0 \, {\rm diag}(\lambda_1^{(0)}, \ldots , \lambda_{{n_1}-1}^{(0)}) \, U_0^*, B_1)$ for an $(n_1-n_0) \times (n_1-n_0)$ unitary matrix $U_0$ and $B_1= (B_{st})_{s,t \in \{ 1, \ldots, k-1\}}$. By the proposed induction assumption, the desired result follows. 
\end{proof}

\begin{proof}
Now we can prove relation \eqref{f3.4} of Theorem \ref{t3.2}: Obviously, the set at the right-hand side of \eqref{f3.4} is a subset of ${\mathcal A}_\sigma(\Delta)$. Conversely, if $A \in {\mathcal A}_\sigma (\Delta)$, it has the form $A=V \, {\rm diag}(\sigma_1^{(0)}, \ldots , \sigma_n^{(0)}) \, W^*$ for some $V,W \in \mathcal U$. Consequently, we observe ${\rm tr}(A \Delta A^*) = {\rm tr}(B) \, {\rm diag}(\lambda_n, \ldots , \lambda_1)$, where $B = U^*W \, {\rm diag}((\sigma_1^{(0)})^2, \ldots , (\sigma_n^{(0)})^2) \, W^*U$ for some $U \in {\mathcal U}_\Delta$. If the minimum \eqref{f3.3} is admitted at $A$, Lemma \ref{l3.3} implies that $B$ has the form $B= \tilde{U} \, {\rm diag}((\sigma_1^{(0)})^2, \ldots , (\sigma_n^{(0)})^2) \tilde{U}^*$ for some $\tilde{U} \in {\mathcal U}_\Delta$. Therefore, we get the equality 
$W^* U' \, {\rm diag}(\sigma_1^{(0)}, \ldots , \sigma_n^{(0)}) = {\rm diag}(\sigma_1^{(0)}, \ldots , \sigma_n^{(0)}) \, W^*U'$ for any $U' \in {\mathcal U}_\Delta$, and hence, 
\[
A=V \, {\rm diag}(\sigma_1^{(0)}, \ldots , \sigma_n^{(0)}) \, W^* U' (U')^* = V W^* U' \, {\rm diag}(\sigma_1^{(0)}, \ldots , \sigma_n^{(0)})  \, (U')^* \, . 
\]
\end{proof}


\begin{cor} \label{c3.4}
   Let $A \in {\mathcal M}$ such that it possesses the singular value decomposition $A = V \, {\rm diag}(\sigma_1, \ldots , \sigma_n) \, W^*$. Let ${\mathcal U}_1$ and ${\mathcal U}_2$ be subsets of $\mathcal U$ such that ${\mathcal U}_1 \cup {\mathcal U}_2 = \mathcal U$. If ${\mathcal A}= \{ AU : U \in {\mathcal U}_1 \} \cup \{ U^*AU : U \in {\mathcal U}_2 \}$, then for any $\Delta$ we have $m(\Delta; {\mathcal A}) = \sum_{j=1}^n \lambda_{n-j+1} \sigma_j^2$ and ${\mathcal A}(\Delta) = \{ AU^*: U^* \in (( W^* {\mathcal U}_1) \cap {\mathcal U}_\Delta) \} \cup \{ UW^* A WU^* : U^* \in ((W^* {\mathcal U}_2) \cap {\mathcal U}_\Delta) \}$. 
\end{cor}

The next consequence of Theorem \ref{t3.2} plays a key role in the present paper. 

\begin{cor} \label{c3.5}
    Let ${\mathcal A} \subseteq {\mathcal M}$. If $\mathcal A$ is weakly unitarily invariant then $\iota(\,\Delta \,;{\mathcal A}) = \inf \{ \sum_{j=1}^n \lambda_{n-j+1} (\sigma_j(A))^2 : A \in {\mathcal A} \}$ for any $\Delta$.
\end{cor}   

\begin{proof}
Let $A \in {\mathcal A}$, and let $A= V \, {\rm diag}(\sigma_1, \ldots , \sigma_n) \, W^*$ be its singular value decomposition. For any $\Delta$ and any $U \in {\mathcal U}_\Delta$, the matrices $AWU^*$ or $UW^*AWU^*$ have the same singular values as $A$ and belong to the set ${\mathcal A}_\sigma(\Delta)$ as described at \eqref{f3.4} for $\mathcal A$ being weakly unitarily invariant as assumed.
\end{proof}

Corollary \ref{c3.5} shows that the number $\iota(\,\Delta \,;{\mathcal A})$ is uniquely defined by the ordered set of eigenvalues of $\Delta$ as long as ${\mathcal A}$ is weakly unitarily invariant. Thus, it is natural to define a function $g_{\mathcal A}$ on $\mathcal D$ by 
\[
   g_{\mathcal A}(x_1, \ldots , x_n) = \iota({\rm diag} (x_1, \ldots , x_n); {\mathcal A}) \, , \,\, (x_1, \ldots , x_n)^\top \in {\mathcal D} \, .
\]   

\begin{pro} \label{p3.6}
  If  $\mathcal A$ is weakly unitarily invariant and $\Delta$ is arbitrary, then
   \begin{equation} \label{f3.7}
       \iota(\Delta; {\mathcal A}) =  g_{\mathcal A}(\lambda_1, \ldots , \lambda_n) \, .
   \end{equation}
   Conversely, if for some set $\mathcal A$ the relation \eqref{f3.7} is true for any $\Delta$ then there exists a unitarily invariant set ${\mathcal A}'$ such that $\iota( \cdot ; {\mathcal A}) = \iota( \cdot ; {\mathcal A}')$. 
\end{pro}

\begin{proof}       
The first assertion is obvious. To prove its converse, note first
\begin{eqnarray} \label{f3.8}
    \, \, \iota( \Delta ; {\mathcal A}) & = & \inf \{ {\rm tr}(A \Delta A^*) : A \in {\mathcal A } \} =  \inf \{ {\rm tr}(VAU \Delta U^*A^*V^*) : A \in {\mathcal A}  \} \\
                                          & = & \inf \{ {\rm tr}(A \Delta A^*) : A \in V {\mathcal A} U \} = \iota( \Delta ; V{\mathcal A}U ) \nonumber
\end{eqnarray}    
for any $U,V \in {\mathcal U}$, by \eqref{f3.7} and because $\lambda_j(\Delta) = \lambda_j(U \Delta U^*)$ for any $j \in \{ 1, \ldots , n\}$. Since the accumulated set ${\mathcal A}'= \bigcup_{V,U \in {\mathcal U}} V {\mathcal A} U$ is unitarily invariant the result follows from Lemma \ref{l2.2}(v).
\end{proof}

Given $g_{\mathcal A}$ we introduce its permutation invariant extension to $[0,+\infty)^n$ denoted by $({g}_{\mathcal A})\tilde{\,}$. We would like to describe the set  $\tilde{{\mathcal G}} = \{ ({g}_{\mathcal A})\tilde{\,}: {\mathcal A} \,\, {\rm is} \, {\rm unitarily} \, {\rm invariant} \}$. In particular, $({g}_{\mathcal A})\tilde{\,}$ can be considered as the restriction of $\iota( \cdot ; {\mathcal A})$ to the set of non-negative Hermitian diagonal matrices. Therefore, the function $({g}_{\mathcal A})\tilde{\,}$ is non-negative, positively homogeneous, superadditive, continuous, permutation invariant and concave, so in particular, Schur concave, cf.~\cite[Prop.~3.C.2]{MOA}. As a question, does a function belong to $\tilde{{\mathcal G}}$ provided it has the named properties? The following example shows that the property of Schur concavity instead of concavity does not warrant that a respective function will be an element of $\tilde{{\mathcal G}}$.

\begin{exe} \label{e3.7}
For $(x_1,x_2)^\top\in [0,+\infty)^2$ set $({g}\tilde{\,}(x_1,x_2)) = \max \{ (x_1^{-1}+x_2^{-1})^{-1}, 3/5 \cdot \min \{ x_1, x_2 \} \}$ if $\min \{ x_1, x_2\} > 0$, and $({g}\tilde{\,}(x_1,x_2))=0$ in complementary case. Clearly, $g\tilde{\,}$ is non-negative, positively homogeneous, continuous, permutation invariant. Since both the function within the braces are concave, they are Schur concave, and their maximum is Schur concave as well, cf.~\cite[3.B.1.c]{MOA}. However, $({g}\tilde{\,}(1+2,1+1)) = 6/5 < 3/5+2/3 = ({g}\tilde{\,}(1,1)) + ({g}\tilde{\,}(2,1))$, so $g\tilde{\,}$ is not superadditive, or equivalently, not concave, and does not belong to $\tilde{{\mathcal G}}$. 
\end{exe}

The next assertion is a first result on the structure of $\tilde{{\mathcal G}}$, cf.~Proposition \ref{p5.6} below for its multiplicative analogue. 

\begin{pro}  \label{p3.8}
   Let $\{ g_j : j \in \{ 1, \ldots ,n \} \}$ be arbitrary complex-valued functions on $[0,+\infty)$, and let $g(x_1, \ldots ,x_n) = \sum_{j=1}^n g_j(x_j)$ for $(x_1, \ldots ,x_n)^\top \in {\mathcal D}$. The permutation invariant extension $g\tilde{\,}$ of $g$ belongs to $\tilde{{\mathcal G}}$ if and only if $g(x_1, \ldots ,x_n) = \sum_{j=1}^n b_jx_j$ for some $ (b_n,\ldots,b_1)^\top \in {\mathcal D}$. 
\end{pro}

\begin{proof}
 If $g(x_1, \ldots ,x_n) = \sum_{j=1}^n b_jx_j$ for some $ (b_1,\ldots,b_n)^\top \in {\mathcal D}$ then $g\tilde{\,} \in \tilde{{\mathcal G}}$ follows from Corollary \ref{c3.4} and Proposition \ref{p3.6}. Conversely, assume that $g\tilde{\,} \in \tilde{{\mathcal G}}$. Setting $h_j(x)=g_j(x)-g_j(0)$ for $x \in [0,+\infty)$, $j \in \{ 1,\ldots,n\}$, we obtain $h_j(0)=0$ for any $j \in \{ 1,\ldots,n \}$ and $g(x_1,\ldots,x_n)= \sum_{j=1}^n h_j(x_j)$ because $g(0,\ldots,0)=\sum_{j=1}^n g_j(0) = 0$. Since the function $x \to g(x,0,\ldots,0)=h_1(x)$ is non-negative and positively homogeneous, we get $h_1(x)=b_1x$ for some $b_1 \in [0,+\infty)$. Using the method of inductive conclusion on the set of functions $(x_1,\ldots,x_k) \to g(x_1,\ldots,x_k,0,\ldots,0)$ over $k \in \{ 1,\ldots,n \}$ we find $h_k(x)=b_kx$ for some $b_k \in [0,+\infty)$. Finally, if the resulting vector $(b_n,\ldots,b_1)^\top \not\in {\mathcal D}$, the function $g$ would be not Schur concave according to \cite[Thm.~3.A.3]{MOA}. 
\end{proof}

Let $\phi$ be an element of the set $\mathcal F$ of non-negative and positively homogeneous functions on $\mathcal M$ different from the zero function. We are going to discuss $\iota( \cdot; {\mathcal S}_\phi)$ for the 'unit sphere' ${\mathcal S}_\phi = \{ A \in {\mathcal M}: \phi(A)=1 \}$ of $\phi$. The following facts are repeatedly used in the sequel without explicit arguments:

\begin{lem} \label{l3.9}
   Let $\phi, \psi \in \mathcal F$. Then:
   \begin{itemize}
   \item[(i)] $\iota( \cdot; \phi(A)=1 ) = \iota( \cdot; \phi(A)^a=1 ) = \iota( \cdot; \phi(A) \geq 1)$ for any $a \in {\mathbb R} \setminus \{0\}$.
   \item[(ii)] If $\phi \leq \psi$ then $\iota( \cdot ; {\mathcal S}_\psi) \leq \iota( \cdot ; {\mathcal S}_\phi)$.
   \item[(iii)] If $\phi = a \psi$ for some $a \in (0,+\infty)$ then $\iota( \cdot ; {\mathcal S}_\phi) = a^{-2} \, \iota( \cdot ; {\mathcal S}_\psi)$.
   \item[(iv)] $\iota( \cdot , \max \{ \phi(A), \psi(A)\} =1 ) = \min \{ \iota( \cdot ; {\mathcal S}_\phi), \iota( \cdot ; {\mathcal S}_\psi) \}$.
   \end{itemize}
\end{lem}

We will also use an analogue of Proposition \ref{p2.4}: 

\begin{pro} \label{p3.10}
For any $\Delta$ we have $\iota(\Delta; {\mathcal S}_\phi) \geq \lambda_n$     if and only if
\begin{equation} \label{f3.9}
      \phi( \cdot ) \leq \| \cdot \|_2  \, .
\end{equation}
\end{pro}

\begin{rmk} \label{r3.11}
Clearly, the inequality
\begin{equation}  \label{f3.10}
     \| \cdot \|_2 \leq \phi( \cdot )
\end{equation}
implies the inequality
\begin{equation}    \label{f3.11}
     \iota( \cdot ; {\mathcal S}_\phi ) \leq \iota( \cdot ; \|A\|_2=1 ) \, .
\end{equation}
In comparison to Proposition \ref{p3.10} condition \eqref{f3.11} is not sufficient for condition \eqref{f3.10}. Indeed, if we choose $a \in (1, \sqrt{n})$ and set $\phi(A)=a \sigma_1$ for $A \in \mathcal M$, we get $\phi(I) < \|I\|_2$, so condition \eqref{f3.10}  does not hold. On the other hand, Corollary \ref{c3.12} below and Lemma \ref{l3.9}(iii)  imply that $m(\Delta ; {\mathcal S}_\phi) = a^{-2} \lambda_n < \lambda_n = m(\Delta ; \|A\|_2 =1)$. 
\end{rmk}

A function $\phi$ on $\mathcal M$ is called right unitarily invariant, unitary similarity invariant or unitarily invariant if  $\phi(AU)=\phi(A)$, $\phi(U^*AU) = \phi(A)$ or $\phi(VAU) = \phi(A)$, respectively, for all $A \in \mathcal M$, any $U,V \in \mathcal U$. If $\phi$ is right unitarily invariant or unitary similarity invariant it is called weakly unitarily invariant. Clearly, $\phi$ has one of these properties if and only if its unit sphere ${\mathcal S}_\phi$ has the corresponding property. In accordance with a common normalization of unitarily invariant norms we will often assume additionally that 
\begin{equation} \label{f3.12}
      \phi(E)=1
\end{equation}      
for the minimal projection $E = {\rm diag}(1,0,\ldots,0)$. 

\begin{cor} \label{c3.12}
   Let $\phi \in \mathcal F$ be weakly unitarily invariant and normalized by \eqref{f3.12}. For any $\Delta$ one has $m(\Delta ; {\mathcal S}_\phi) = \lambda_n$ if and only if condition \eqref{f3.9} is satisfied.
\end{cor}
   
\begin{proof}
Apply Proposition \ref{p3.10} and the fact that ${\rm tr}(EU^*\Delta UE) = \lambda_n$, as well as $EU^* \in {\mathcal S}_\phi$ or $UEU^* \in {\mathcal S}_\phi$ for any $U \in {\mathcal U}_\Delta$, in case $\phi$ is normalized by \eqref{f3.12} and weakly unitarily invariant.
\end{proof}
 
In contrast to the preceding corollary there exist unitarily invariant functions $\phi \in \mathcal F$ satisfying the conditions \eqref{f3.10}, \eqref{f3.12}, and nevertheless $m( \Delta ; {\mathcal S}_\phi ) < \lambda_n$ for all invertible $\Delta$, cf.~Proposition \ref{p4.2} below. Even more, we do not know any other function $\phi$ different from $\| \cdot \|_2$ that satisfies 
\begin{equation} \label{f3.13}
      \iota(\Delta ; {\mathcal S}_\phi ) = \lambda_n
\end{equation}
for any $\Delta \in {\mathcal I}$. We give an example where \eqref{f3.13} appears for some invertible $\Delta$.
       
\begin{exe}
Let $n=2$. For $a \in (0,1)$ set $\psi(A) = a (\sigma_1+\sigma_2)$ and $\phi(A) = \max \{ \psi( \cdot ), \| \cdot \|_2 \}$. If $a$ is close to $1$,  there exists an $A \in \mathcal M$ with $\phi(A) > \|A\|_2$. Hence, $\phi$ satisfies \eqref{f3.10} and differs from $\| \cdot \|_2$. By Lemma \ref{l3.9}(iv) and (iii) and by Proposition \ref{p4.2} below, $m(\Delta ; {\mathcal S}_\phi ) = \min \{ m(\Delta ; {\mathcal S}_\psi ), m(\Delta ; \|A\|_2=1) \} = \min \{ a^{-2} \lambda_1 \lambda_2 (\lambda_1+\lambda_2)^{-1}, \lambda_2\}$, what equals to $\lambda_2$ if and only if $\lambda_1 \geq a^2(1-a^2)^{-1} \lambda_2$. 
\end{exe}

By Lemma \ref{l2.3} one obtains a large class of functions $\phi \in \mathcal F$ satisfying $m(\Delta ; {\mathcal S}_\phi )=0$ if $\Delta$ is not invertible. 

\begin{lem} \label{l3.14}
   Let $\phi \in \mathcal F$ be weakly unitarily invariant. The equality $m(\Delta ; {\mathcal S}_\phi )=0$ is true for all non-invertible $\Delta$ if and only if there exists an $A \in \mathcal M$ such that 
   \begin{equation} \label{f3.14}
        \phi(A) > 0 \,\, {\rm and} \,\,  {\rm rk}(A)=1 \, .
   \end{equation}
\end{lem}

If \eqref{f3.14} is not satisfied the value $m(\Delta ; {\mathcal S}_\phi )$ can be strictly greater than zero for any $\Delta \in ({\mathcal M}^\geq \setminus \{ 0 \} )$.  If, e.g., $\phi(A)=\sigma_n$ for $A \in \mathcal M$, then $m(\Delta ; {\mathcal S}_\phi ) = \sum_{j=1}^n \lambda_j$ by Corollary \ref{c3.5}. Note, however, that the realization of condition \eqref{f3.14} is not necessary for the equality $\iota(\Delta ; {\mathcal S}_\phi ) = 0$ to hold for any non-invertible $\Delta$, cf.~Proposition \ref{p5.1} below.


\section{Properties of $m( \cdot ; {\mathcal S}_{\| \cdot \|})$ as an important particular case }
\label{s4}

We are going to deal with the calculation of the minimum $m( \cdot ; {\mathcal S}_\phi)$ if $\phi( \cdot ) = \| \cdot \|$ is a norm on $\mathcal M$ which has one of the three properties of unitary invariance under consideration. Taking into account Lemma \ref{l3.14} we can and shall assume throughout this section that $\Delta$ is invertible to compute $m( \Delta ; {\mathcal S}_{\| \cdot \|})$.

A unitarily invariant norm $\| \cdot \|$ on $\mathcal M$ is called a Q-norm if there exists a unitarily invariant norm $\| \cdot \| \widehat{\,\,}$ such that $\| A \|^2=\| A^*A \| \widehat{\,\,}$ for $A \in \mathcal M$, cf.~\cite[Def.~IV.2.9]{Bh}. Let $\| \cdot \|_1$ denote the Schatten 1-norm, i.e. $\| A \|_1 = \sum_{j=1}^n \sigma_j$. Since for a Q-norm $\| \cdot \|$ normalised by \eqref{f3.12}, the chain of inequalities $\| A \|^2 = \| A^*A \| \widehat{\,\,} \leq \| A^*A \|_1 = \| A \|_2^2$ is true, we derive the next assertion by Corollary \ref{c3.12}: 

\begin{pro} \label{p4.1}
If $\| \cdot \|$ is a Q-norm then $m( \Delta ; {\mathcal S}_{\| \cdot \|}) = \| E \|^{-2} \lambda_n$ for any $\Delta$. 
\end{pro}

Let us mention that there exists a unitarily invariant norm  $\| \cdot \|$ satisfying $\| \cdot \| \leq \| \cdot \|_2$ and $\| E \| = 1$, which is not a Q-norm. For example, if the norm defined by $\| A \| = 1/2 \cdot (\| A \|_2 + \sigma_1(A))$ would be a Q-norm, there would exist a symmetric gauge function $g$ such that $g(x_1,\ldots,x_n)=1/4 \cdot ((\sum_{j=1}^n x_j)^{1/2} + x_1)^2$ for $(x_1,\ldots,x_n)^\top \in \mathcal D$. Elementary calculations reveal that $g$ does not satisfy the triangle inequality and, thus, $g$ cannot be a symmetric gauge function. 

For ${\mathbf t}=(t_1,\ldots,t_n)^\top \in [0,+\infty)^n \setminus \{ 0 \}$ and $p \in (0,+\infty)$ define a unitarily invariant function $\phi_{{\mathbf t},p} \in \mathcal F$ by
\begin{equation*} 
   \phi_{{\mathbf t},p}(A) = \left( \sum_{j=1}^n t_j\sigma_j^p \right)^{1/p} \, , \,\, A \in \mathcal M \, .
\end{equation*}

If ${\mathbf t} \in \mathcal D$ and $p \in [1,+\infty)$, then $\phi_{{\mathbf t},p}$ is a norm, cf.~\cite[Ex.~IV.1.19]{Bh} for a more general result. Conversely, if $ \phi_{{\mathbf t},p}$ is a norm, its restriction to the set of non-negative Hermitian diagonal matrices can be identified with a convex and permutation invariant function on $[0,+\infty)^n$. By \cite[3.C.2]{MOA} this function is Schur convex, and \cite[Thm.~3.A.3]{MOA} shows ${\mathbf t} \in \mathcal D$ and $p \in [1,+\infty)$. To simplify denotation we use ${\mathcal S}_{{\mathbf t},p}$ instead of  ${\mathcal S}_{\phi_{{\mathbf t},p}}$ in the sequel.

\begin{pro} \label{p4.2}
    Let ${\mathbf t} \in {\mathcal D}$, $t_1 > 0$ and $\Delta \in {\mathcal I}$. 
    \begin{itemize}
    \item[(i)] If $p \in [2,+\infty)$ then $m(\Delta ; {\mathcal S}_{{\mathbf t},p}) = t_1^{-2/p}\lambda_n$. 
    \item[(ii)] If $p \in (0,2)$ then 
            \begin{equation} \label{f4.1}
                   m(\Delta ; {\mathcal S}_{{\mathbf t},p}) = \left( \sum_{j=1}^n \left( t_j^{2/p} \lambda_{n-j+1}^{-1} \right)^{p/(2-p)} \right)^{(p-2) / p}
             \end{equation}
             and ${\mathcal S}_{{\mathbf t},p}(\Delta)$ is the set of all $A$ of the form
             \begin{equation} \label{f4.2}
                   A = \left( \sum_{j=1}^n \left(  t_j^{2/p} \lambda_{n-j+1}^{-1} \right)^{p/(2-p)} \right)^{-1/p} \cdot V \,  {\rm diag}((t_1 \lambda_n^{-1})^{1/(2-p)}, \ldots ,( t_n \lambda_1^{-1})^{1/(2-p)} ) \, U^* 
             \end{equation}
              with $V \in {\mathcal U}$, $U \in {\mathcal U}_\Delta$.
     \end{itemize}
\end{pro}
              
\begin{proof}
(i) We know ${\phi}_{{\mathbf t},p}$ is a Q-norm if $p \in [2,+\infty)$. Therefore, the assertion follows from Proposition \ref{p4.1}. 

(ii) Let $p \in (0,2)$. If $a_j, b_j$ are positive real numbers for $j \in \{1,\ldots,n\}$ and let $\sum_{j=1}^k a_j^p = 1$ for some $k$, H\"olders inequality for exponents $2/p$ and $2/(2-p)$ yields $1 = \sum_{j=1}^k b_j^{-p/2}b_j^{p/2}a_j^p \leq \left( \sum_{j=1}^k b_ja_j^2\right)^{p/2} \left( \sum_{j=1}^k b_j^{p/(p-2)} \right)^{(2-p)/2}$ or equivalently, 
\begin{equation} \label{f4.3}
             \sum_{j=1}^k b_ja_j^2 \geq \left( \sum_{j=1}^k b_j^{p/(p-2)} \right)^{(p-2)/p} \, .
\end{equation}
Set $k = \max \{ j : t_j > 0 \}$ and 
\begin{equation}   \label{f4.4}
       a_j=t_j^{1/p} \sigma_j \, , \,\, b_j=\lambda_{n-j+1}t_j^{-2/p} \, .
\end{equation}         
Let ${\phi}_{{\mathbf t},p}(A) = 1$. By Corollary \ref{c3.5} and \eqref{f4.3} we have the inequality
\begin{eqnarray} \label{f4.5}
       {\rm tr}(A \Delta A^*) & \geq  & \sum_{j=1}^n \lambda_{n-j+1} \sigma_j^2 \geq \sum_{j=1}^k  \lambda_{n-j+1} \sigma_j^2 \\
       & \geq &        \left( \sum_{j=1}^k \left( t_j^{2/p}\lambda_{n-j+1}^{-1} \right)^{p/(2-p)} \right)^{(p-2)/p}  \nonumber \\
       & = &        \left( \sum_{j=1}^n \left( t_j^{2/p}\lambda_{n-j+1}^{-1} \right)^{p/(2-p)} \right)^{(p-2)/p} \, . \nonumber
\end{eqnarray}
If $A$ has the form \eqref{f4.2} the chain of inequalities 
\begin{equation} \label{f4.6}
       t_1 \lambda_n^{-1} \geq \ldots \geq t_n \lambda_1^{-1}
\end{equation}
implies that $\sigma_\ell = \left( t_\ell \lambda_{n-\ell+1}^{-1} \right)^{1/(2-p)} \cdot \left( \sum_{j=1}^n ( t_j^{p/2} \lambda_{n-j+1}^{-1})^{p/(2-p)} \right)^{-1/p}$ for $\ell$ in 
\linebreak[4]
$\{ 1,\ldots,n\}$. Elementary computations show that ${\phi}_{{\mathbf t},p}(A)=1$ and ${\rm tr}(A \Delta A^*)$ is equal to the right-hand side of \eqref{f4.1}.              

Conversely, let $A$ be such that ${\phi}_{{\mathbf t},p}(A)=1$ and ${\rm tr}(A \Delta A^*)$ is equal to the right-hand side of \eqref{f4.1}. Then all inequalities of \ref{f4.5} must be equalities. In particular, an equality in the second inequality is equivalent to ${\rm rk}(A) \leq k$. Now, assume for a moment that $\Delta = {\rm diag}(\lambda_n, \ldots ,\lambda_1)$. If $A=V \, {\rm diag}(\sigma_1,\ldots,\sigma_k, 0,\ldots,0)\, W^*$ is a singular value decomposition of $A$, one has ${\rm tr}(A \Delta A^*) = \sum_{\ell=1}^k \sigma_\ell^2 \sum_{j=1}^n \lambda_{n-j+1} |w_{j\ell}|^2 \geq \sum_{j=1}^k \sigma_j^2 \lambda_{n-j+1} |w_{jj}|^2$ with equality if and only if $|w_{j\ell}|=\delta_{j\ell}$, the Kronecker symbol, for any $\ell \in \{ 1,\ldots,k\}$, $j \in \{ 1,\ldots,n\}$. Consequently, 
\[
A= V \, {\rm diag}({\rm e}^{{\rm i}\,{\rm arg}(w_{11})}, \ldots , {\rm e}^{{\rm i}\,{\rm arg}(w_{kk})}, 0,  \ldots , 0) \cdot {\rm diag}(\sigma_1,\ldots,\sigma_k,0,\ldots,0) \, .
\]
To continue we have to take into account the equality condition in H\"older's inequality. Indeed, \eqref{f4.3} is an equality if and only if $(b_j^{2/p}a_j^p)^{2/p} = c (b_j^{-2/p})^{2/(2-p)}$, or equivalently, $a_j = c^{1/2} b_j^{1/(p-2)}$ for some $c \in (0,+\infty)$, any $j \in \{ 1,\ldots,k\}$. By \eqref{f4.4} one obtains $\sigma_j=c^{1/2} \lambda_{n-j+1}^{1/(p-2)} t_j^{1/(2-p)}$, and the condition $\phi_{{\mathbf t}, p}(A)=1$ yields 
$c=\left( \sum_{j=1}^n (t_j^{2/p} \lambda_{n-j+1}^{-1})^{p/(2-p)}\right)^{-2/p}$, giving the form \eqref{f4.2} for $A$. Finally, if $\Delta \in {\mathcal I}$ is arbitrary, use the form of $AU$, $U \in {\mathcal U}_\Delta$, by the result just proven and obtain \eqref{f4.2}. 
\end{proof}

From the proof of assertion (ii) of the preceding proposition we see that formula \eqref{f4.1} is true for some $\Delta$ as soon as the chain of inequalities \eqref{f4.6} is valid, even if ${\mathbf t} \not\in {\mathcal D}$. Note that the condition ${\mathbf t} \in {\mathcal D}$ is equivalent to the fact that \eqref{f4.6} is true for all $\Delta \in {\mathcal I}$. This observation allows to give the assertion of Proposition \ref{p4.2}(ii) a slightly more general expression. 

\begin{cor} \label{c4.3}
   Let $p \in (0,2)$ and $\Delta \in \mathcal I$. Let $0=n_0 < n_1 < \ldots < n_k \leq n$  be non-negative integers and let ${\mathbf t}=(t_1,\ldots,t_n)^\top \in [0,+\infty)^n$ be such that 
 \begin{equation} \label{f4.7}
           t_{n_1}(n_1-n_0)^{-1} \geq t_{n_2}(n_2-n_1)^{-1} \geq \ldots \geq t_{n_k} (n_k-n_{k-1})^{-1}
\end{equation}
and let $t_s=0$ for $s \not=n_\ell$, $\ell \in \{ 1,\ldots,k\}$. Then 
\begin{equation} \label{f4.8}
           m(\Delta; {\mathcal S}_{{\mathbf t},p}) = \left(  \sum_{\ell=1}^k t_{n_\ell}^2 \left(  \sum_{s=n_{\ell-1}}^{n_\ell-1} \lambda_{n-s} \right)^{p/(p-2)} \right)^{(p-2)/p} \, .
\end{equation}
\end{cor}
           
\begin{proof}
If $s \not= n_\ell$ then $\sigma_s \geq \sigma_{s+1}$ is the only condition on $\sigma_s$. Taking $\sigma_s$ as small as possible we have to minimize $\sum_{\ell=1}^k \left(\sum_{s=n_{\ell-1}}^{n_\ell-1} \lambda_{n-s} \right)\sigma_{n_\ell}^2$ under the condition $\sum_{\ell=1}^k t_{n_\ell}\sigma_{n_\ell}^p =1$. The inequalities \eqref{f4.6} take the form $t_{n_1} \left(\sum_{s=n_0}^{n_1-1} \lambda_{n-s} \right)^{-1} \geq t_{n_2} \left(\sum_{s=n_1}^{n_2-1}  \lambda_{n-s} \right)^{-1}  \geq \ldots \geq t_{n_k} \left(\sum_{s=n_{k-1}}^{n_k-1}  \lambda_{n-s} \right)^{-1} $ and are satisfied for any vector $(\lambda_1,\ldots,\lambda_k)^\top \in {\mathcal D}_+$ if and only if \eqref{f4.7} is true. Now, Proposition \ref{p4.2}(ii) yields \eqref{f4.8}.
\end{proof}

For $\mathbf t$ as in the corollary, a conclusion analogous to that one of Proposition \ref{p4.2}(i) fails for $p \in [2,+\infty)$. Indeed, to prove the inequality $\phi( \cdot ) \leq \| \cdot \|_2$ we have to verify the inequality $\left( \sum_{\ell=1}^k t_{n_\ell} \sigma_{n_\ell}^p \right)^{1/p} \leq \left(\sum_{\ell=1}^k (n_\ell-n_{\ell-1}) \sigma_{n_\ell}^2\right)^{1/2}$ which cannot be obtained, in general, replacing $\sigma_{n_l}^2$ by the particular values $(\sigma_{n_\ell}')^2 =(n_\ell-n_{\ell-1}) \sigma_{n_\ell}^2$, since  the order of $\sigma_{n_\ell}'$ might be destroyed. Only a weaker extension of Proposition \ref{p4.2}(i) can be stated.

\begin{cor}  \label{c4.4}
   Let $p \in [2,+\infty)$ and $\mathbf t$ and $\Delta$ as in Corollary \ref{c4.3}. If additionally $n_1-n_0 \geq n_2-n_1 \geq \ldots \geq n_k-n_{k-1}$    then $m(\Delta ; {\mathcal S}_{{\mathbf t}, p}) = t_{n_1}^{-2/p} \sum_{s=1}^{n_1 - 1} \lambda_{n-s+1}$. 
\end{cor}   

An example shows that for general $\mathbf t$ one cannot expect to describe $m(\Delta ; {\mathcal S}_{{\mathbf t}, p})$ by a single formula.

\begin{exe}  \label{e4.5}
    Let $p \in (0,+\infty)$, $t \in (0,+\infty)$ and $\phi(A)= (\sigma_1^p+t \sigma_n^p)^{1/p}$. According to Corollary \ref{c3.5} we have to minimize the function $f$ defined by $f(x_1,x_2) = \lambda_nx_1 + \sum_{j=1}^{n-1} \lambda_j x_2$ for $x_1 \geq x_2 \geq 0$ with the additional condition $x_1^q + tx_2^q =1$, where $q=p/2$. A solution can be easily found looking at an $x_1,x_2$ Cartesian system of coordinates. If $q \in [1,+\infty)$, or equivalently $p \in [2,+\infty)$, the function
    \begin{equation} \label{f4.9}
        x_2= t^{-1/q}(1-x_1^q)^{1/q}  \,  ,  \,  x_1 \in [0,1]
    \end{equation}
is concave, and the minimum of $f$ on this function \eqref{f4.9} is attained at $(1,0)$  or at $((1+t)^{-1/q}, (1+t)^{-1/q})$. Since $f((1,0))= \lambda_n$ and $f((1+t)^{-1/q}, (1+t)^{-1/q})=(1+t)^{-1/q} \sum_{j=1}^n \lambda_j$ we get  $m(\Delta ; {\mathcal S}_{{\mathbf t}, p})=\lambda_n$ if $(\sum_{j=1}^n \lambda_j )^p \geq (1+t)^2 \lambda_n^p$ and  $m(\Delta ; {\mathcal S}_{{\mathbf t}, p})=(1+t)^{-2/p} \sum_{j=1}^n \lambda_j$ if $(\sum_{j=1}^n \lambda_j )^p < (1+t)^2 \lambda_n^p$. For $q \in (0,1)$, or equivalently $p \in (0,2)$, the function \eqref{f4.9} is convex. If $\Delta$ is such that $t \lambda_n \leq \sum_{j=1}^{n-1} \lambda_j$, or equivalently, $\lambda_n^{-1} \geq t (\sum_{j=1}^{n-1} \lambda_j)^{-1}$, by the arguments in the proof of Proposition \ref{p4.2}(ii) we obtain  $m(\Delta ; {\mathcal S}_{{\mathbf t}, p})= \left( \lambda_n^{-p/(2-p))} + t^{2/(2-p)} \left(\sum_{j=1}^{n-1} \lambda_j\right)^{-p/(2-p)} \right)^{(p-2)/p}$. If $t \lambda_n > \sum_{j=1}^{n-1} \lambda_j$, or equivalently $-t^{-1} > -\lambda_n (\sum_{j=1}^{n-1} \lambda_j )^{-1}$, the minimum of $f$ on the curve of the function \eqref{f4.9} is attained at the point $((1+t)^{-1/q}, (1+t)^{-1/q})$ since the derivative of the function \eqref{f4.9} at this point equals to $-t^{-1}$. Thus,  $m(\Delta ; {\mathcal S}_{{\mathbf t}, p}) = (1+t)^{-1} \sum_{j=1}^{n} \lambda_j$ in that case. 
\end{exe}

Recall, that for $C \in \mathcal M$ the $C$-numerical radius $w_C(A)$ of a matrix $A$ is defined by
\begin{equation} \label{f4.10}
    w_C(A) = \max \{ | {\rm tr}(UCU^*A) | : U \in {\mathcal U} \} \, .
\end{equation}
The function $w_C$ is a unitary similarity invariant seminorm on $\mathcal M$ and a norm if and only if ${\rm tr}(C) \not= 0$ and $C$ is not a scalar multiple of $I$, cf.~\cite[IV.4.4]{Bh}. To simplify denotation set ${\mathcal S}_{w_C} = {\mathcal S}_C$.

\begin{pro} \label{p4.6}
     Let $C \in {\mathcal M} \setminus \{ 0 \}$. For any $\Delta \in {\mathcal I}$ we have
     \begin{equation} \label{f4.11}
           m(\Delta ; {\mathcal S}_C) = \left( \sum_{j=1}^n \sigma_j(C) \, \lambda_{n-j+1}^{-1} \right)^{-1} \, .
     \end{equation}
\end{pro}

\begin{proof}
Recall that
\begin{equation} \label{f4.12}
        w_C(A) \leq \sum_{j=1}^n \sigma_j(C) \sigma_j(A) 
\end{equation}
for $C \in {\mathcal M} \setminus \{ 0 \}$, cf.~\cite[20.B.1]{MOA}. For $A \in \mathcal M$, set $|A|=(A^*A)^{1/2}$ and denote the set of all $X \in \mathcal U$ satisfying $w_C(X|A|) = \sum_{j=1}^n \sigma_j(C) \sigma_j(A)$ by ${\mathcal U}^{(A)}$. We show that ${\mathcal U}^{(A)}$ is not empty. Indeed, let $C= V_C \, {\rm diag}(\sigma_1(C), \ldots , \sigma_n(C) ) \, W_C^*$ be a singular value decomposition of $C$. If $V \in \mathcal U$ is such that $|A| = V \, {\rm diag}(\sigma_1(A), \ldots , \sigma_n(A)) \, V^*$, set $U=VV_C$ at the right-hand side of \eqref{f4.10} and set $X=VV_C^*W_CV^*$. We get ${\rm tr}(UCU^*X|A|) = \sum_{j=1}^n \sigma_j(C) \sigma_j(A)$, and hence, $X \in {\mathcal U}^{(A)}$ by \eqref{f4.12} and $\sigma_j(A) = \sigma_j(X|A|)$ for any $j \in \{ 1,\ldots,n\}$. For $A \in {\mathcal M} \setminus \{0\}$ set $a=\sum_{j=1}^n \sigma_j(C) \sigma_j(A)$. Let $\tilde{{\mathcal A}}$ be the set of all $\tilde{A} \in \mathcal M$ of the form $\tilde{A} = a^{-1}X|A|$, $X \in {\mathcal U}^{(A)}$, $A \in {\mathcal M} \setminus \{0\}$. Since $w_C(\tilde{A})=a^{-1} w_C(X|A|) = 1$, the set $\tilde{\mathcal A}$ is a subset of ${\mathcal S}_C$. Inequality \eqref{f4.12} yields $a \geq 1$, what implies that ${\rm tr}(\tilde{A}\Delta \tilde{A}) = a^{-2} {\rm tr}(A\Delta A) \leq {\rm tr}(A \Delta A)$ and
        \begin{equation} \label{f4.13}
              m( \cdot ; {\mathcal S}_C) = m( \cdot ; \tilde{\mathcal A})
        \end{equation}
by Lemma \ref{l2.2}(i). Since 
\[ 
\sum_{j=1}^n \sigma_j(C) \sigma_j(\tilde{A}) = a^{-1} \sum_{j=1}^n \sigma_j(C) \sigma_j(X|A|) =a^{-1} \sum_{j=1}^n \sigma_j(C) \sigma_j(A) = 1 \, , 
\]
the set $\tilde{A}$ is also a subset of the set ${\mathcal A}' = \{ A' \in {\mathcal M} : \sum_{j=1}^n \sigma_j(C) \sigma_j(A')=1 \}$, so
\begin{equation} \label{f4.14}
             m( \cdot ; {\mathcal A}') \leq m( \cdot ; {\mathcal S}_C)
\end{equation}
by \eqref{f4.13}. If $A' \in {\mathcal A}'$ and $V \in \mathcal U$ are such that $|A'| = V \, {\rm diag}(\sigma_1(A'), \ldots , \sigma_n(A')) \, V^*$, the matrix $A= VV_C^*W_C \, {\rm diag}(\sigma_1(A'), \ldots , \sigma_n(A')) \, V^*$ has the same singular values as $A'$. We calculate ${\rm tr}(VV_C^*CV_CV^* A) =     \sum_{j=1}^n \sigma_j(C) \sigma_j(A') = 1$, and therefore, $w_C(A) = 1$ by \eqref{f4.12}. Since $A^*A=V \, {\rm diag}((\sigma_1(A')^2, \ldots, \sigma_n(A')^2) \, V^*$ we get $|A|^2 = |A'|^2$, and hence, ${\rm tr}(A \Delta A) = {\rm tr} (A' \Delta A')$ and $m( \cdot ; {\mathcal S}_C) = m( \cdot ; {\mathcal A}')$ by \eqref{f4.14} and Lemma \ref{l2.2}(i). To complete the proof apply Proposition \ref{p4.2}(ii) for $p=1$ and $t_j = \sigma_j(C)$, $j \in \{1,\ldots,n\}$.
\end{proof}

\begin{cor} \label{c4.7}
   If
   \begin{equation} \label{f4.15}
       w_C(E)=1 
   \end{equation}
   the following conditions are equivalent:
   \begin{itemize}
   \item[(i)] $C = z P$ for a one-dimensional orthogonal projection $P$ and some $z \in \mathbb C$ with $|z|=1$.
   \item[(ii)] $w_C$ is the numerical radius $w_C(A) = \sup_{\{ {\mathbf x} : \|{\mathbf x}\|_2=1 \} } | \langle {\mathbf x}, A ({\mathbf x}) \rangle |$ for any $A \in \mathcal M$.
   \item[(iii)] $w_C(A) \leq \|A\|_2$ for any $A \in \mathcal M$.
   \end{itemize}
\end{cor}

\begin{proof}
The implications (i)$\to$(ii)$\to$(iii) are obvious. \newline
(iii)$\to$(i): By Corollary \ref{c3.12}, assertion (iii) is equivalent to the equality $m(\Delta ; {\mathcal S}_C) = \lambda_n$ for any $\Delta$. Therefore, \eqref{f4.11} implies ${\rm rk}(C)=1$ and $\sigma_1(C)=1$. According to Schur's triangulization theorem there exist $U \in {\mathcal U}$ and ${\mathbf a}=(a_1,\ldots,a_n)^\top \in {\mathbb C}^n$ such that $C=U \left( \begin{array}{c} {\mathbf a}^\top \\ 0 \end{array} \right) U^*$ and $\| {\mathbf a}\|_2=1$. If there would exist $j \in \{ 2,\ldots,n \}$ with $a_j \not= 0$, the inequality 
\begin{equation} \label{f4.16}
     \left| {\mathbf y}^* C {\mathbf y} \right| < 1
\end{equation}
would follow for any ${\mathbf y} \in {\mathbb C}^n$ with $\| {\mathbf y} \|_2=1$.  To see this let ${\mathbf x}=(x_1,\ldots,x_n)^\top = U^* {\mathbf y}$. Then 
$ \left| {\mathbf y}^* C {\mathbf y} \right| = | a_1 |< \| {\mathbf a} \|_2=1$ if $| x_1 | =1$, and $ \left| {\mathbf y}^* C {\mathbf y} \right| = | x_1^* \sum_{j=1}^n a_jx_j | \leq | x_1 | \| {\mathbf a} \|_2 \| {\mathbf x} \|_2 < 1$ if $| x_1 | < 1$. Since \eqref{f4.16} is a contradiction to \eqref{f4.15}, assertion (i) is proved. 
\end{proof}   

Let $\| \cdot \|$ be a unitarily invariant norm. Recall, that its dual norm $\| \cdot \|^D$ is defined by $\| A \|^D = \max \{ | {\rm tr}(AC^*)| : \| C \| = 1 \}$, cf.~\cite[5.4.12]{HJ2}. Another consequence of Proposition \ref{p4.2} is an alternative extremal representation of $m( \cdot ; {\mathcal S}_{\| \cdot \|})$ in terms of the dual norm of $\| \cdot \|$.

\begin{pro} \label{p4.8}
    For invertible $\Delta \in {\mathcal M}^\geq$ we have 
    \begin{equation*}
    m(\Delta ;  {\mathcal S}_{\| \cdot \|}) = \min \left\{ \left( \sum_{j=1}^n (\sigma_j(X))^2 \lambda_{n-j+1}^{-1} \right)^{-1} : X \in {\mathcal M} \, , \, \,\| X \|^D=1 \right\} \, .
    \end{equation*}
\end{pro}

\begin{proof}
For $X \in {\mathcal M}$ denote by ${\mathbf s}(X)$ the vector of the singular values of $X$ ordered decreasingly, ${\mathbf s}(X) = (\sigma_1(X),\ldots,\sigma_n(X) )^\top$. Applying \cite[Thm.~3.5.5]{HJ2}, Lemma \ref{l2.2}(v) and Proposition \ref{p4.2}(ii) with $p=1$, one obtains 
\begin{eqnarray*}
m(\Delta ;  {\mathcal S}_{\| \cdot \|}) & = &  m(\Delta ; \{ A \in {\mathcal M} : \max \{ \phi_{{\mathbf s}(X), 1}(A) :   \,\| X \|^D=1 \} = 1 \}) \\
          & = & \min \left\{ m(\Delta; \{ A \in {\mathcal M} : {\phi}_{{{\mathbf s}(X)},1} (A)=1 \}) : \,\| X \|^D=1 \right\} \\
          & = & \min \left\{ \left( \sum_{j=1}^n (\sigma_j(X))^2 \lambda_{n-j+1}^{-1} \right)^{-1} :  X \in {\mathcal M} \, , \,\| X \|^D=1 \right\} \, . 
\end{eqnarray*}         
\end{proof}    

We finish this section giving a positive lower bound for $m(\Delta ; {\mathcal S}_{\| \cdot \|})$, $\Delta \in \mathcal I$. Let $g$ denote the symmetric gauge function of a unitarily invariant norm $\| \cdot \|$. Recall, that the Ky Fan $k$-norm is defined as the sum of the first $k$ singular values of a given $m \times n$ matrix $A$, i.e.~$\|A\|_{(k)} = \sum_{i=1}^k s_i(A)$ with $k \leq \min \{ m,n\}$, where the singular values are assumed to be ordered in decreasing order, cf.~\cite[p.~92]{Bh}, \cite{Fan,Watson,Ding}.

\begin{pro} \label{p4.9}
     For invertible $\Delta \in \mathcal M^\geq$ we have the estimate
     \begin{equation}  \label{f4.17}
          m(\Delta ; {\mathcal S}_{\| \cdot \|} ) \geq ( \| E \| g(\lambda_1^{-1},\ldots,\lambda_n^{-1}))^{-1} \, .
     \end{equation}     
     In particular, let $\| \cdot \|$ be a norm of the form $\phi_{{\mathbf t},p}$. If $p=1$ and all elements of $\mathbf t$ different from $0$ coincide, i.e.~if $\| \cdot \|$ is a positive multiple of a Ky Fan $k$-norm, \eqref{f4.17} is an equality for any invertible $\Delta$. Otherwise, \eqref{f4.17} is a strong inequality for all invertible $\Delta$.
\end{pro}

\begin{proof}
Let $\| A \|=1$. The Schatten 1-norm is known to be the largest unitarily invariant norm satisfying the normalization \eqref{f3.12}. By Theorem \ref{t3.2} and by \cite[IV.1.6]{Bh} we get 
   \begin{eqnarray*}
        g(\lambda_1^{-1}, \ldots , \lambda_n^{-1}) \, {\rm tr}(A \Delta A^*)
             & \geq & g(\lambda_1^{-1}, \ldots , \lambda_n^{-1})\sum_{j=1}^n \lambda_{n-j+1} \sigma_j^2 \\
             & \geq &  g(\lambda_1^{-1}, \ldots , \lambda_n^{-1}) \| E \|^{-1}   g(\lambda_1 \sigma_n^2, \ldots , \lambda_n \sigma_1^2) \\
             & \geq & \| E \|^{-1}   (g(\sigma_1, \ldots , \sigma_n))^2 = \| E \|^{-1}\, . 
   \end{eqnarray*}
So \eqref{f4.17} follows. Also, \eqref{f4.1} shows that \eqref{f4.17} is an equality for any $\Delta \in \mathcal I$ if $\| \cdot \|$ is a positive multiple of a Ky Fan $k$-norm. 

Now, assume that $\| \cdot \|$ is a norm of the form $\phi_{{\mathbf t}.p}$ and \eqref{f4.17} is an equality for some $\Delta \in {\mathcal I}$. If $p \in [2,+\infty)$ we get $(t_1^{1/p}(\sum_{j=1}^n t_j \lambda_{n-j+1}^{-p})^{1/p})^{-1} = t_1^{-2/p} \lambda_n$ by Proposition \ref{p4.2}(i). Setting $x_j = \lambda_{n-j+1}^{-1}$ for $j \in \{ 1,\ldots, n \}$ this equality indicates that there exists a solution $(x_1,\ldots,x_n)^\top \in {\mathcal D}_+$ of the equation $\sum_{j=1}^n t_j^px_j = t_1 x_1$, what is equivalent to $t_2=0$ since ${\mathbf t} \in \mathcal D$. Similarly, if $p \in [1,2)$ then there is a solution $(x_1,\ldots,x_n)^\top \in {\mathcal D}_+$ of the equation
\begin{equation} \label{f4.18}
    \left( \sum_{j=1}^n t_j^{2 / (2-p)} x_j^{p/(2-p)} \right)^{2-p} = t_1 \sum_{j=1}^n t_j x_j^p 
\end{equation}
by Proposition \ref{p4.2}(ii). The inequality $(\sum_{j=1}^n t_1^{2/(2-p)} x_j^{p / (2-p)} )^{(2-p)} \leq \sum_{j=1}^n t_j^2 x_j^p$ holds since $(2-p) \in (0,1]$ . Equality is satisfied if and only if $p=1$ or $t_2=0$. In case $p=1$ the equality \eqref{f4.18} can be rewritten as $\sum_{j=1}^n t_j^2 x_j = t_1 \sum_{j=1}^n t_j x_j$, and we get $t_j^2 = t_1 t_j$ for any $j \in \{1,\ldots,n \}$, i.e. $t_1= \ldots = t_k > 0$ and $t_{k+1}= \ldots =  t_n=0$ for some $k \in \{ 1,\ldots, n-1 \}$.        
\end{proof}   

Because of the preceding proposition one might guess that for all unitarily invariant norms the inequality \eqref{f4.17} should be an equality either for any $\Delta \in {\mathcal I}$ or for no $\Delta \in {\mathcal I}$. We give an example where \eqref{f4.17} is an equality for a non-trivial subset of ${\mathcal M}^\geq \cap \mathcal I$. The problem of a suitable description of the set of all unitarily invariant norms such that the inequality \eqref{f4.17} is an equality for all or for no $\Delta \in \mathcal I$, respectively,  remains to be open. 

\begin{exe} \label{e4.10}
Set $n=2$ and define a symmetric gauge function $g(x_1,x_2)= \max \{ \sqrt{2}x_1, x_1+x_2\}$,  $(x_1,x_2)^\top \in \mathcal D$. Then $\| E \| = \sqrt{2}$ for the corresponding unitarily invariant norm $\| \cdot \|$. Thus, the right-hand side of \eqref{f4.17} is equal to 
\begin{eqnarray*}
    \| E \|^{-1}g(\lambda_1^{-1},\lambda_1^{-1})^{-1} 
    & = & \sqrt{2}^{-1} (\max \{ \sqrt{2}\lambda_1^{-1}, \lambda_1^{-1}+\lambda_2^{-1}\} )^{-1} \\
    & = & \sqrt{2}^{-1} \min \{ \sqrt{2}^{-1} \lambda_1, (\lambda_1^{-1}+\lambda_2^{-1})^{-1} \} = \frac{1}{2} \lambda_1
\end{eqnarray*}
if and only if $\sqrt{2}^{-1} \lambda_1 \leq (\lambda_1^{-1}+\lambda_2^{-1})^{-1}$, or equivalently, $\lambda_1 \leq (\sqrt{2}-1) \lambda_2$. On the other hand, $m(\Delta ; {\mathcal S}_{\| \cdot \|}) = \min \{ \lambda_1 / 2, (\lambda_1^{-1}+\lambda_2^{-1})^{-1} \}$ by Lemma \ref{l3.9}(iv) and Proposition \ref{p4.2}(ii). It is not hard to see that $m(\Delta ; {\mathcal S}_{\| \cdot \|}) = \lambda_1/2$ if and only if $\lambda_1=\lambda_2$. Therefore, the inequality \eqref{f4.17} is an equality if and only if $\lambda_1=\lambda_2$ or $\lambda_1 \geq (\sqrt{2}-1) \lambda_2$.
\end{exe}

\section{Further examples and related inequalities}
\label{s5}

First, we present an amplification of a result by Helson and Lowdenslager \cite[Thm.~8]{HL} (cf. Thm.~\ref{t1.1} above). For ${\mathbf t}=(t_1,\ldots,t_n)^\top \in (0,+\infty)^n$ set $t = \sum_{j=1}^n t_j$ and ${\mathcal S}_{\mathbf t} = \{ A \in {\mathcal M} : \prod_{j=1}^n \sigma_j^{t_j} = 1 \}$. 

\begin{pro} \label{p5.1}
Let ${\mathbf t}=(t_1,\ldots,t_n)^\top \in (0,+\infty)^n$ and $\Delta \in {\mathcal M}^\geq$.
\begin{itemize}
\item[(i)] If $\Delta$ is not invertible then $\iota(\Delta ; {\mathcal S}_{\mathbf t}) = 0$ and the infimum is attained if and only if $\Delta=0$. 
\item[(ii)] Suppose, ${\mathbf t} \in {\mathcal D}_+$. For invertible $\Delta$,
\begin{equation} \label{f5.1}
     m(\Delta ; {\mathcal S}_{\mathbf t}) = t \left( \prod_{j=1}^n (\lambda_{n-j+1} t_j^{-1})^{t_j} \right)^{1/t} \, .
\end{equation}
and 
\begin{eqnarray*}
 \quad \quad {\mathcal S}_{\mathbf t}(\Delta) = \{ A \in {\mathcal M} : A=t^{-1}  m(\Delta ; {\mathcal S}_{\mathbf t})^{1/2} \, V \, {\rm diag}((t_1\lambda_n^{-1})^{1/2}, \ldots , (t_n \lambda_1^{-1})^{1/2}) U^* \, , & &\\
V \in {\mathcal U}, U \in {\mathcal U}_\Delta \} \, . &  & 
\end{eqnarray*}
\end{itemize}
\end{pro}

To prepare the proof we need the following lemma to apply Corollary \ref{c3.5}:

\begin{lem} \label{l5.2}
    Let ${\mathbf t}=(t_1,\ldots,t_n)^\top \in (0,+\infty)^n$ and $(a_1,\ldots,a_n)^\top \in (0,+\infty)^n$. The minimum $\mu$ of the function $f(x_1,\ldots,x_n) = \sum_{j=1}^n a_j x_j$ for $(x_1,\ldots,x_n)^\top \in (0,+\infty)^n$ with the condition 
\begin{equation} \label{f5.2}
     \prod_{j=1}^n x_j^{t_j} = 1
\end{equation}
equals to
\begin{equation} \label{f5.3}
     \mu = t \left( \prod_{j=1}^n (a_j t_j^{-1})^{t_j} \right)^{1/t}\, ,
\end{equation}
and the minimum $\mu$ is attained at a unique point 
\begin{equation} \label{f5.4}
     (x_1^{(0)}, \ldots , x_n^{(0)})^\top= t^{-1} (t_1a_1^{-1},\ldots,t_n a_n^{-1})^\top  \, .
\end{equation}
\end{lem}     

\begin{proof}
Since $\lim_{x_j \to +\infty} f(x_1,\ldots,x_n) = +\infty$ for any $j \in \{ 1,\ldots,n\}$ a minimum exists. We compute it using the Lagrange mollifier method. Set $F(x_1,\ldots,x_n,\lambda) = f(x_1,\ldots,x_n) - \lambda  \, \left(\prod_{j=1}^n x_j^{t_j} - 1\right)$ and compute the stationary points solving the system $\frac{{\partial} F}{{\partial} x_k} (x_1,\ldots,x_n) = a_k - \lambda t_kx_k^{-1} \cdot \prod_{j=1}^n x_j^{t_j} =0$, $k \in \{ 1,\ldots,n \}$, under \eqref{f5.2}. One obtaines $a_kx_k=\lambda t_k$, what implies $\prod_{j=1}^n x_j^{t_j} \, \prod_{j=1}^n a_j^{t_j} = \lambda^t \, \prod_{j=1}^n t_j^{t_j}$. Therefore, $\lambda= \left( \prod_{j=1}^n (t_j a_j^{-1})^{t_j} \right)^{1/t}$ showing that \eqref{f5.4} is the only stationary point. Since a minimum exists and any point of local minimum is a stationary point, relation \eqref{f5.3} follows by elementary computations.
\end{proof}

\begin{proof} (of Proposition \ref{p5.1})
(i) Choose two sequences $(b_k)$ and $(c_k)$ of positive numbers such that $b_k^{t_1} c_k^{t-t_1} = 1$ and $\lim_{k \to +\infty} c_k =0$. If $U \in{\mathcal U}_\Delta$ and $A_k= {\rm diag}(b_k,c_k,\ldots,c_k) \, U^*$ then the sequence $(A_k)$ meets the condition of Lemma \ref{l2.3}. If $\Delta$ is not invertible and different from zero, the space ${\mathcal N}(\Delta)$ is a proper subspace of ${\mathbb C}^n$. Since all matrices of ${\mathcal S}_{\mathbf t}$ are invertible,  ${\mathcal S}_{\mathbf t}(\Delta) = \emptyset$ by Lemma \ref{l2.3}.

(ii) Note, that  $A \in {\mathcal S}_{\mathbf t}$ if and only if $\prod_{j=1}^n (\sigma_j^2)^{t_j}=1$. Set $x_j=\sigma_j$ for $j \in \{ 1,\ldots,n \}$ and apply Corollary \ref{c3.5} and Lemma \ref{l5.2} choosing $a_j=\lambda_{n-j+1}$. In particular, observe that ${\mathbf t} \in {\mathcal D}_+$ implies that the numbers $\sigma_j^{(0)}=(x_j^{(0)})^{1/2}$, $j \in \{ 1,\ldots,n \}$, satisfy the chain of inequalities $\sigma_1^{(0)} \geq \ldots \geq \sigma_n^{(0)}$ by \eqref{f5.4}. This is a necessary condition by convention \eqref{f2.1}.
\end{proof}

\begin{rmk} \label{r5.3}
If ${\mathbf t} \not\in {\mathcal D}_+$, the proof of Proposition \ref{p5.1}(ii) reveals that \eqref{f5.1}is true for any matrix $\Delta$, whose eigenvalues satisfy the chain of inequalities \eqref{f4.6}. If $\Delta$ is such that \eqref{f4.6} fails, formula \eqref{f5.1} cannot be applied. For example, if $n=2$ and \eqref{f4.6} is not true, one has to move on the curve $x_1^{t_1} x_2^{t_2}=1$ to the right as far as one reaches a point $(x_1,x_2)$ with $x_1=x_2$, i.e., at the point $(1,1)$. One has $m(\Delta ; {\mathcal S}_{\mathbf t})= \lambda_1+\lambda_2$, in this case. 
\end{rmk}

\begin{rmk} \label{r5.4}
If ${\mathbf t} \in{\mathcal D}$ is such that $t_k > 0$ and $t_{k+1}=0$ for some $k \in \{ 1,\ldots,n-1\}$, the condition $\prod x_j^{t_j}=1$ can be interpreted in two ways. Either we set $0^0=1$ or we consider $0^0$ as not defined. In the first case one can assume $\sigma_{k+1}=0$, in the second case all singular values have to be positive, but $\sigma_{k+1}$ can be arbitrary small. In both these cases we have only to apply Lemma \ref{l5.2} with $n=k$. The only difference is that in the first case the minimum is admitted as soon as ${\rm rk}(\Delta) \geq k$, whereas in the second case the infimum is not attained similarly as in Proposition \ref{p5.1}(i).
\end{rmk}

Using Proposition \ref{p5.1} and the following lemma a multiplicative analogue of Proposition \ref{p3.8} can be proved. 

\begin{lem} \label{l5.5}
   Let $\{ g_i : i \in \{ 1,\ldots,n \}\}$ be measurable complex-valued functions on $(0,+\infty)$ and let $g(x_1,\ldots,x_n)=\prod_{j=1}^n g_j(x_j)$, where $(x_1,\ldots,x_n)^\top \in {\mathcal D}_+$. Let $g \not= 0$ be non-negative, positively homogeneous and Schur concave. Suppose,
   \begin{equation}\ \label{f5.6}
       g(x_1,\ldots,x_n) \leq g(y_1,\ldots,y_n) \,\, {\rm if} \,\, x_j \leq y_j \,\, {\rm for} \,\, j \in \{1,\ldots,n\} \, .
   \end{equation}       
Then $g$ is of the form $g(x_1,\ldots,x_n) = c \cdot \prod_{j=1}^n x_j^{t_j}$ for some constant $c \in (0,+\infty)$ and a vector $(t_n,\ldots,t_1)^\top \in {\mathcal D}$ such that $t=\sum_{j=1}^n t_j = 1$.
\end{lem}   

\begin{proof}
Since $g$ is non-zero and positively homogeneous, the functions $g_j$, $ j=1,\ldots,n$,  do not have zero values. Replacing these functions $g_j$ by functions
\linebreak[4]
 ${\rm exp} \left( {\mathbf i} \, {\rm arg}(g_j(1)-{\rm arg}(g_j) \right) (g_j(1))^{-1} g_j$ for $j \in \{1,\ldots,n\}$, we can assume that all the functions $g_j$ are positive and $g_j(1)=1$. Condition \eqref{f5.6} implies that these functions $g_j$ are increasing. Thus, a non-negative measurable function $h$, can be defined by $h_j(x)=\log g_j(e^x)$, for $x \in [0,+\infty)$. Also, define $h$ by $h(x_1,\ldots,x_n)= \log \, g(e^{x_1},\ldots,e^{x_n})$, $(x_1,\ldots,x_n)^\top \in \mathcal D$. So, $h_j(0)=0$,
\begin{equation} \label{f5.7}
      h(x_1,\ldots,x_n)=\sum_{j=1}^n h_j(x_j) 
\end{equation}
and 
\begin{equation}   \label{f5.8}
     h(x_1+y,\ldots,x_n+y) = y + h(x_1,\ldots,x_n) \, , \,\, (x_1,\ldots,x_n)^\top \in {\mathcal D} \, ,\, y \in [0,+\infty) \, ,
\end{equation}           
since $g$ is positively homogeneous. Now, we prove by induction on $k \in \{1,\ldots,n\}$  that for some $t_k \in [0,+\infty)$,
\begin{equation} \label{f5.9}
     h_k(x)=t_k x \, , \,\, x \in[0,+\infty) \, .
\end{equation}
Setting $x_1=0$ in \eqref{f5.8} one obtains
\begin{equation} \label{f5.10}
     \sum_{j=2}^n h_j(y) = y - h_1(y)
\end{equation}
by \eqref{f5.7}. If $x_1=x$, $x_2=0$ then $h_1(x+y) + \sum_{j=2}^n h_2(y) = h(x+y,y,\ldots,y) = y+h_1(x)$. Consequently, \eqref{f5.10} yields $h_1(x+y)=h_1(x)+h_1(y)$, and measurability of $h_1$ gives \eqref{f5.9} for $k=1$. 

Now, let $k \in \{1,\ldots,n-1\}$. Again, setting $x_1=0$ in \eqref{f5.8}, by the induction hypothesis one has
\begin{equation}   \label{f5.11}
     \sum_{j=k+2}^n h_j(y) = y - \sum_{j=1}^k t_j y -h_{k+1}(y) \, ,
\end{equation}
where $\sum_{j=k+2}^n h_j(0) =0$    if $k=n-1$. Setting $x_1=\ldots=x_{k+1}=x$ and $x_{k+2}=0$ we can conclude
\begin{eqnarray*}
\begin{split}
   & \sum_{j=1}^k t_j(x+y) +h_{k+1}(x+y) + \sum_{j=k+2}^n h_j(y) = \\
    &\quad \quad \quad = h(x+y,\ldots,x+y,y,\ldots,y) = y + \sum_{j=1}^k t_jx + h_{k+1}(x)
\end{split}
\end{eqnarray*}
(with argument change at position $k+2$), so $h_{k+1}(x+y) = h_{k+1}(x) + h_{k+1}(y)$ by \eqref{f5.11}. Measurability of $h_{k+1}$ completes the proof of \eqref{f5.9}.  Thus, $g(x_1,\ldots,x_n) = g(1,\ldots,1) \, \prod_{j=1}^n x_j^{t_j}$ for $(x_1,\ldots,x_n)^\top \in {\mathcal D} \cap [1,+\infty)^n$ and, moreover, for any vector $(x_1,\ldots,x_n)^\top \in \mathcal D$ since $g$ is positively homogeneous. Additionally, $\sum_{j=1}^n t_j = 1$. Finally,  $\frac{{\partial} g(x_1,\ldots,x_n)}{{\partial} x_k} = t_kx_k^{-1} \, \prod_{j=1}^n x_j^{t_j}$, and Schur concavity of $g$ yields $(t_n,\ldots,t_1)^\top \in \mathcal D$, cf.~\cite[Thm.~3.A.3]{MOA}. 
\end{proof}

\begin{pro} \label{p5.6}
   Let $\{ g_j : j = 1,\ldots,n \}$ be arbitrary complex-valued functions on $[0,+\infty)$ and set
   \begin{equation} \label{f5.12}
        g(x_1,\ldots,x_n) = \prod_{j=1}^n g_j(x_j) \, , \,\, (x_1,\ldots,x_n)^\top \in {\mathcal D} \, .
   \end{equation}
   The permutation invariant extension $g\tilde{\,}$ belongs to the set $\tilde{{\mathcal G}} = \{ ({g}_{{\mathcal A}})\tilde{\,} : {\mathcal A} \,\, {\it is} \, {\it unitarily}$ ${\it invariant} \}$ if and only if 
   \begin{equation} \label{f5.13}
        g(x_1,\ldots,x_n) = c \cdot \prod_{j=1}^n x_j^{t_j}     
   \end{equation}     
   for some constant $c \in [0,+\infty)$ and $(t_n,\ldots,t_1)^\top \in \mathcal D$ such that $\sum_{j=1}^n t_j=1$.
\end{pro}

\begin{proof}
The 'if-part' follows from Proposition \ref{p5.1}. Conversely, let $g$ have the form \eqref{f5.12} and let $g\tilde{\,} \in \tilde{{\mathcal G}}$. Then the function $g$, and hence, all the functions $g_j$ are continuous. Lemma \ref{l5.5} implies that $g$ has the form \eqref{f5.13}. 
\end{proof}   

Similarly to Remark \ref{r5.4}, the assertion of Proposition \ref{p4.2} can be extended to negative exponents in two ways.

\begin{pro}   \label{p5.7}
    Let $p \in (0,+\infty)$ and $(t_1,\ldots,t_n)^\top \in \mathcal D$, $t_1 > 0$.
    \begin{itemize}
    \item[(i)]  For any $\Delta \in {\mathcal M}^\geq$ one has  $m(\Delta ; \sum_{j=1}^n t_j (\sigma_j^{\dagger})^p = 1 ) = t_1^{2/p} \lambda_n$. 
    \item[(ii)] Let $r=\max \{ j : t_j >0 \}$ and ${\mathcal A} = \{ A \in {\mathcal M} : {\rm rk}(A) \geq r \,\, {\rm and} \,\,\sum_{j=1}^r t_j \sigma_j^{-p} = 1 \}$. If $\Delta \in {\mathcal M}^\geq$ is invertible then 
    \begin{equation} \label{f5.14}
         m(\Delta ; {\mathcal A}) = \left(\sum_{j=1}^r \left(\lambda_{n-j+1} t_j^{2/p} \right)^{p / (p+2)} \right)^{(p+2)/p}   \, .
    \end{equation}
    If $\Delta \in {\mathcal M}^\geq$ is not invertible then $\iota(\Delta ; {\mathcal A})=0$ and the infimum zero is attained if and only if ${\rm rk}(\Delta) \leq n-r$.  
    \end{itemize}   
\end{pro}

\begin{proof}
(i) Since $(\sum_{j=1}^n t_j (\sigma_j(A)^{(\dagger)})^p)^{-1/p} \leq t_1^{-1/p} \| A \|_2$ and ${\rm tr}(A \Delta A^*)=t_1^{2/p} \lambda_n$ if $A={\rm diag}(t_1^{1/p},0,\ldots,0) \, U^*$. $U \in {\mathcal U}_\Delta$, the result follows from Corollary \ref{c3.12}.

(ii) Let $\Delta \in {\mathcal I}$. Since there are no conditions on $\sigma_{r+1},\ldots,\sigma_n$ we set $\sigma_{r+1}=0$. With $q=p/2$ we apply Corollary \ref{c3.5}. We have to minimize the function $f$ defined by
     \begin{equation} \label{f5.15}
           f(x_1,\ldots,x_r) = \sum_{j=1}^r \lambda_{n-j+1} x_j \,\, , \,\, {\rm while} \,\, x_1 \geq \ldots \geq x_r > 0 \,\, {\rm and} \,\, \sum_{j=1}^r t_j x_j^{-q}=1 \, .
     \end{equation}
To compute the stationary points of the function $F(x_1,\ldots,x_r,\lambda) = f(x_1,\ldots,x_r) - \lambda \cdot (\sum_{j=1}^r t_j x_j^{-q} - 1 )$ we solve the system $\frac{{\partial} F}{{\partial} x_k} (x_1,\ldots,x_r) = \lambda_{n-k+1} + \lambda q t_kx_k^{-q-1} =0$. We get the solution $x_k = (\lambda_n t_k \lambda_{n-k+1}^{-1}t_1^{-1})^{1/(q+1)}x_1$ for $k \in \{ 1,\ldots,r\}$. Taking into account \eqref{f5.15} we get the equality $x_1^{-q} \sum_{k=1}^r t_k(\lambda_{n-k+1}t_1 \lambda_n^{-1}t_k)^{q/(q+1)} = 1$, leading to the coordinates $x_j^{(0)} = \left( \sum_{k=1}^r (\lambda_{n-k+1} t_k^{1/q} )^{q / (q+1)} \right)^{1/q} (t_j\lambda_{n-j+1}^{-1})^{1/(q+1)}$, $j \in \{ 1,\ldots,r \}$, of a unique stationary point and a corresponding value of $m(\Delta ; {\mathcal A})$ as in \eqref{f5.14} replacing $q$ by $p/2$. 

Suppose $\Delta \not\in {\mathcal I}$. Then there exists a sequence $(A_k)$ satisfying the conditions of Lemma \ref{l2.3}, which implies $\iota(\Delta ; {\mathcal A})=0$. There exists a matrix $A \in {\mathcal A}$ such that ${\mathcal R}(\Delta) \subseteq  {\mathcal N}(A)$ if and only if the dimension of ${\mathcal N}(A)$ is not smaller than $r$, or, what is the same, ${\rm rk}(\Delta) \leq n-r$.
\end{proof}

For $k \in \{ 1,\ldots,n \}$ and $r \in \{ 1,\ldots,k \}$ let $E_{r,k}$ be the $r$th elementary symmetric harmonic polynomial of $k$ variables, i.e. $E_{r,k}(x_1,\ldots,x_k) = \sum_{j \in J} x_{j_1} \ldots x_{j_r}$ with $(x_1,\ldots,x_k)^\top \in [0,+\infty)^k$ and $J$ the (index) set of all strictly increasing $r$-tuples $(j_1,\ldots,j_r)$ selected from $\{ 1,\ldots,k \}$. Since $m(\Delta ; E_{1,n}(\sigma_1^2,\ldots,\sigma_n^2)=1)=\lambda_n$ by Proposition \ref{p4.2}(i) and since $m(\Delta ; E_{n,n}(\sigma_1^2,\ldots,\sigma_n^2)=1) = n \cdot {\rm det}(\Delta^{1/n})$ by Proposition \ref{p5.1}(ii) one could try to find a formula for $m(\Delta ; E_{r,n}(\sigma_1^2,\ldots,\sigma_n^2)=1)$, $r \in \{ 2,\ldots,n-1\}$. We do not know a general solution. However, Proposition \ref{p5.9} below gives an answer in case $r=2$. Again, we start with an auxiliary extremal problem. Let $(a_k)_{k \in \mathbb N}$ be an increasing sequence of positive real numbers. For $n \geq 2$ let $f_n(x_1,\ldots,x_n)=\sum_{j=1}^n a_jx_j$ with $(x_1,\ldots,x_n)^\top \in [0,+\infty)^n$. Consider the following extremal problem ${\mathcal E}_n$:
\begin{equation} \label{f5.16}
     {\rm Find} \,\, \mu_n = \min f_n(x_1,\ldots,x_n) \,\, {\rm under} \, {\rm the} \, {\rm condition} \,\, E_{2,n}(x_1,\ldots,x_n)=1 \, .
\end{equation}     
Since the problem ${\mathcal E}_{n-1}$ can be considered as the problem ${\mathcal E}_n$ restricted to the part $\{ (x_1,\ldots,x_{n-1},0)^\top : x_1,\ldots,x_{n-1} \in [0,+\infty)^{n-1} \}$ of its boundary we get $\mu_n \leq \mu_{n-1}$ for any $n$. Moreover, some computations show that ${\mathcal E}_2$ has exactly one stationary point $(x_1^{(2)},x_2^{(2)})$ satisfying $\mu_2 = f_2(x_1^{(2)},x_2^{(2)}) = 2 \sqrt{a_1a_2}$. 

\begin{lem} \label{l5.8}
   Let $n \geq 3$. The problem ${\mathcal E}_n$ has at most one stationary point $(x_1^{(n)},\ldots,x_n^{(n)} )$. The stationary point exists if and only if 
   \begin{equation}   \label{f5.17}
       \min \left\{ \sum_{j=1}^n a_j - (n-1) a_n , \left(\sum_{j=1}^n a_j \right)^2  - (n-1)\sum_{j=1}^n a_j^2 \right\} > 0 \, .
   \end{equation}
In this case
   \begin{equation}   \label{f5.18}
        \mu_n= f(x_1^{(n)},\ldots,x_n^{(n)}) = \sqrt{\frac{2}{n-1}} \left( \left(\sum_{j=1}^n a_j \right)^2 - (n-1) \sum_{j=1}^n a_j^2 \right)^{1/2} \, .
   \end{equation}     
\end{lem}
   
\begin{proof}
Let $F(x_1,\ldots,x_n,\lambda) = f(x_1,\ldots,x_n) - \lambda (E_{2,n}(x_1,\ldots,x_n)-1)$, add the $n$ equations $\frac{{\partial} F}{{\partial} x_k} (x_1,\ldots,x_n,\lambda) = a_k-\lambda (\sum_{j=1}^n x_j - x_k)=0$, $k \in \{ 1,\ldots,n \}$, and subtract from this sum the $k$-th equality multiplied by $(n-1)$. One obtains $\sum_{j=1}^n a_j - (n-1)a_k = (n-1) \lambda x_k$ what transforms to
\begin{equation}   \label{f5.19}
    x_k = ((n-1)\lambda)^{-1} \left(\sum_{j=1}^n a_j - (n-1)a_k \right) \, .
\end{equation}
Since $x_n$ should be positive, necessarily $\sum_{j=1}^n a_j - (n-1)a_n \not= 0$. Moreover, if this expression would be negative then necessarily $\lambda < 0$, leading to an contradiction $x_1 <0$ since  $\sum_{j=1}^n a_j - (n-1)a_n > 0$. Consequently, the condition 
$\sum_{j=1}^n a_j - (n-1)a_n  > 0$ is necessary for the existence of a stationary point.  Using the equality $E_{2,n} (x_1,\ldots,x_n) = \frac{1}{2} \, \left((\sum_{j=1}^n x_j)^2 - (n-1) \sum_{j=1}^n x_j^2 \right)$ one easily obtains
\begin{equation}   \label{f5.20_1}
    \lambda^2 =(2 (n-1)^2)^{-1} \left((\sum_{j=1}^n a_j)^2 - \sum_{j=1}^n a_j^2 \right) \, ,
\end{equation}
by \eqref{f5.16} and \eqref{f5.19}, what implies inequality $\left((\sum_{j=1}^n a_j)^2 - (n-1) \sum_{j=1}^n a_j^2\right) > 0$ to be also necessary for the existence of a stationary point.

Conversely, if \eqref{f5.17} is satisfied, the computations above, particularly \eqref{f5.20_1} and \eqref{f5.19}, reveal the existence of a unique stationary point $(x_1^{(n)},\ldots,x_n^{(n)})^\top$. One can compute $f_n(x_1^{(n)},\ldots,x_n^{(n)}) = \sqrt{2/(n-1)} \left((\sum_{j=1}^n a_j)^2 - (n-1) \sum_{j=1}^n a_j^2 \right)^{1/2}$. We have only to demonstrate finally that this term is the minimum $\mu_n$ of the extremal problem ${\mathcal E}_n$. Since, obviously, the minimum of ${\mathcal E}_n$ at the boundary $\{ (x_1,\ldots,x_n)^\top : x_j=0 \,\, {\rm for} \, {\rm some} \,\, j \in \{1,\ldots,n)\}\}$ equals the minimum at the set $\{ (x_1,\ldots,x_{n-1},0): (x_1,\ldots,x_{n-1})^\top \in [0,+\infty)^{n-1} \}$ it suffices to show that for all $n \geq 3$ the following fact is true: If ${\mathcal E}_n$ has a stationary point $(x_1^{(n)}, \ldots , x_n^{(n)} )$ then ${\mathcal E}_{n-1}$ has a stationary point $(x_1^{(n-1)}, \ldots , x_{n-1}^{(n-1)} )$ as well, and 
\begin{equation}  \label{f5.20_2}
    f_n(x_1^{(n)},\ldots,x_n^{(n)}) \leq f_{n-1}(x_1^{(n-1)},\ldots,x_{n-1}^{(n-1)}) \, .
\end{equation}    
We proceed by induction on $n$: If ${\mathcal E}_3$ has a stationary point $(x_1^{(3)},x_2^{(3)},x_3^{(3)} \}$ then 
\begin{eqnarray*}
          f_3(x_1^{(3)},x_2^{(3)},x_3^{(3)}) & = & ((a_1+a_2+a_3)^2 -2(a_1^2+a_2^2+a_3^2))^{1/2} \\
                                                               & = &  (4a_1a_2 - (a_1+a_2-a_3)^2)^{1/2} \\
                                                            & \leq & 2 \sqrt{a_1a_2} = f_2(x_1^{(2)},x_2^{(2)}) \, .
\end{eqnarray*}
For arbitrary $n$, the inequality $a_{n-1} \leq a_n$ yields
\begin{equation} \label{f5.21}
      \sum_{j=1}^n a_j - (n-1) a_n \leq \sum_{j=1}^{n-1} a_j - (n-2) a_{n-1} \, .
\end{equation}
Moreover, the valid inequality $0 \leq (\sum_{j=1}^{n-1} a_j - (n-2) a_n)^2$ can be transformed into the equivalent inequality
\begin{equation}    \label{f5.22}
      ( \sum_{j=1}^n a_j )^2 - (n-1) \sum_{j=1}^n a_j^2 \leq \frac{n-1}{n-2} \left(  (\sum_{j=1}^{n-1} a_j)^2 - (n-2) \sum_{j=1}^{n-1} a_j^2\right) \, .
\end{equation}
Therefore, by \eqref{f5.18}, \eqref{f5.21} and \eqref{f5.22} the existence of a stationary point of ${\mathcal E}_n$ is sufficient for the existence of a stationary point of ${\mathcal E}_{n-1}$. Finally, by \eqref{f5.19} and \eqref{f5.22} the inequality \eqref{f5.20_2} follows. 
\end{proof}   

For $n \geq 2$ set ${\mathcal A} = \{ A \in {\mathcal M}: E_{2,n}(\sigma_1^2, \ldots , \sigma_n^2)=1 \}$.                                    

\begin{pro} \label{p5.9}
(i) Let $\Delta \in {\mathcal M}^\geq$ be invertible. If 
\[
\min \left\{ \sum_{j=1}^n \lambda_{n-j+1} - (n-1)\lambda_1, \left( \sum_{j=1}^n \lambda_{n-j+1}\right)^2 - (n-1) \sum_{j=1}^n \lambda_{n-j+1}^2 \right\} >0\, ,
\]
 then we can calculate 
 \[
 m(\Delta ; {\mathcal A}) = \sqrt{\frac{2}{n-1}}  \left( \left( \sum_{j=1}^n \lambda_{n-j+1}\right)^2 - (n-1) \sum_{j=1}^n \lambda_{n-j+1}^2 \right)^{1/2} \, .
 \]
 If for $k \in \{ 3,\ldots,n \}$, the inequalities
   \begin{eqnarray}   \label{f5.23}
         & & \\
        \min \left\{ \sum_{j=1}^k \lambda_{n-j+1} - (k-1) \lambda_{n-k+1} , \left(\sum_{j=1}^k \lambda_{n-j+1}\right)^2 - (k-1) \sum_{j=1}^k \lambda_{n-j+1}^2 \right\} \leq 0 \, ,& &  \nonumber \\
        \min \left\{ \sum_{j=1}^{k-1} \lambda_{n-j+1} - (k-2) \lambda_{n-k+2} , \left(\sum_{j=1}^{k-1} \lambda_{n-j+1}\right)^2 - (k-2) \sum_{j=1}^{k-1} \lambda_{n-j+1}^2 \right\} > 0 &  & \nonumber
   \end{eqnarray}
   are satisfied, then $m(\Delta ; {\mathcal A}) = \sqrt{\frac{2}{k-2}}  \left( ( \sum_{j=1}^{k-1} \lambda_{n-j+1})^2 - (k-2) \sum_{j=1}^{k-1} \lambda_{n-j+1}^2 \right)^{1/2}$.
   
(ii) If $\Delta$ is not invertible then $\iota(\Delta ; {\mathcal A}) = 0$. The infimum zero is attained if and only if ${\rm rk}(\Delta) \leq n-2$. 
\end{pro}   

\begin{proof}
(i) Let $\Delta \in \mathcal I$. By Lemma \ref{l5.8}, if the extremal problem ${\mathcal E}_n$ does not have a stationary point then its minimum is attained at the set $\{ ( x_1,\ldots,x_{n-1},0)^\top : (x_1,\ldots,x_{n-1})^\top \in [0,+\infty)^{n-1} \}$, what leads to the extremal problem ${\mathcal E}_{n-1}$. If the extremal problem ${\mathcal E}_n$ has a stationary point $(x_1^{(n)} ,\ldots , x_n^{(n)})^\top$ then the coordinates of it satisfy the chain of inequalities $x_1^{(n)} \geq \ldots \geq x_n^{(n)} \geq 0$ by \eqref{f5.19}. So, assertion (i) follows by an application of Lemma \ref{l5.8}. 

(ii) Let $\Delta \not\in \mathcal I$. There exists a sequence $(A_k)$ satisfying the conditions of Lemma \ref{l2.3}. Since ${\rm rk}(A) \geq 2$ for $A \in \mathcal A$, (ii) follows by an application of the second assertion of Lemma \ref{l2.3}.
\end{proof}

We enrich this collection of examples giving another application of Proposition \ref{p4.2}. Set $f_{r,k}(A) = E_{r,k} (a_{11},\ldots,a_{kk})$ and $\phi_{r,k}(A) = \max \{ f_{r,k} (U^*AU) : U  \in {\mathcal U} \}$ for ${\mathcal A} \in {\mathcal M}^\geq$. 

\begin{pro} \label{p5.10}
For $k \in \{ 1,\ldots,n \}$, $r \in \{ 1,\ldots,k \}$ and $\Delta \in {\mathcal M}^\geq$ invertible, we have 
\begin{equation*}
    m( \Delta ; \{ A \in {\mathcal M}^\geq : f_{r,k} (A)=1 \} ) = \frac{1}{k^2} \left( \begin{array}{c} k \\  r \end{array} \right)^{2/r} \cdot \left(\sum_{j=1}^k \lambda_{n-j+1}^{-1} \right)^{-1} \, .
\end{equation*}
\end{pro}

\begin{proof}
The description of the minimum set ${\mathcal S}_{{\mathbf t},p}(\Delta )$ in Proposition \ref{p4.2}(ii) gives the equality 
\begin{equation} \label{f5.24}
       m(\Delta ; {\mathcal S}_{{\mathbf t},p} ) = m(\Delta ; \{ A \in {\mathcal M}^\geq : \phi_{{\mathbf t},p}(A) = 1 \}) \, .
\end{equation}
Since  $ f_{r,k}(A) = \left( \begin{array}{c} k \\ r \end{array} \right) \left( \frac{1}{k} \, \sum_{j=1}^k \lambda_j(A) \right)^r$, cf.~\cite[Cor.~after Thm.~2]{MM}, one has 
\begin{eqnarray*}
\begin{split}
& m(\Delta ; \{  A \in  {\mathcal M}^\geq : f_{r,k}(A)=1 \}) = \\
   & \quad \quad =  m \left(\Delta ; \left\{ A \in {\mathcal M}^\geq : \sum_{j=1}^k \lambda_j(A) = k \left( \begin{array}{c} k \\ r \end{array} \right)^\frac{-1}{r} \right\} \right) \\
   & \quad \quad =  \frac{1}{k^2} \left( \begin{array}{c} k \\ r \end{array} \right)^{2/r}  \cdot m\left(\Delta ; \left\{ A \in {\mathcal M}^\geq : \sum_{j=1}^k \lambda_j(A)=1 \right\}\right) \, . 
\end{split}   
\end{eqnarray*}
Now, apply \eqref{f5.24} and Proposition \ref{p4.2}(ii). 
\end{proof}

It is well-known that extremal representations of matrix-valued functions are a rich source of inequalities. We want to sketch ways to obtain inequalities from the results above illustrating them by examples. 

A first set of inequalities can be immediately derived from concavity and superadditivity of the function $\iota( \cdot : {\mathcal A})$. For example, selecting ${\mathbf t}=(1,\ldots,1)^\top$ and $p=1$ in Proposition \ref{p4.2}(ii) one obtains
\begin{eqnarray} \label{f5.25}
& & \\
       \frac{E_{n,n}(\lambda_1(\Delta_1+\Delta_2),\ldots,\lambda_n(\Delta_1+\Delta_2))}{E_{n-1,n}(\lambda_1(\Delta_1+\Delta_2),\ldots,\lambda_n(\Delta_1+\Delta_2))} & \geq &
       \frac{E_{n,n}(\lambda_1(\Delta_1),\ldots,\lambda_n(\Delta_1))}{E_{n-1,n}(\lambda_1(\Delta_1),\ldots,\lambda_n(\Delta_1))} \nonumber \\
       & & +
       \frac{E_{n,n}(\lambda_1(\Delta_2),\ldots,\lambda_n(\Delta_2))}{E_{n-1,n}(\lambda_1(\Delta_2),\ldots,\lambda_n(\Delta_2))} \nonumber
\end{eqnarray}  
for $\Delta_1,\Delta_2 \in {\mathcal M}^\geq \cap \mathcal I$. In particular, if $\Delta_1={\rm diag}(x_1,\ldots,x_n)$, $\Delta_2 = {\rm diag}(y_1,\ldots,y_n)$ we obtain superadditivity of harmonic mean or the equivalent inequality
\begin{equation} \label{f5.26}
       \frac{E_{n,n}(x_1+y_1,\ldots,x_n+y_n)}{E_{n-1,n}(x_1+y_1,\ldots,x_n+y_n)} \geq 
        \frac{E_{n,n}(x_1,\ldots,x_n)}{E_{n-1,n}(x_1,\ldots,x_n)} +
         \frac{E_{n,n}(y_1,\ldots,y_n)}{E_{n-1,n}(y_1,\ldots,y_n)}
\end{equation}
for $(x_1,\ldots,x_n)^\top, (y_1,\ldots,y_n)^\top  \in (0,+\infty)^n$

Conversely, if the inequality \eqref{f5.26} is true the function $\frac{E_{n,n}}{E_{n-1,n}}$ is concave, and hence, Schur concave on $(0,+\infty)^n$, cf.~\cite[Prop.~3.C.2]{MOA}. Since the vector $(\lambda_1(\Delta_1+\Delta_2),\ldots,\lambda_n(\Delta_1+\Delta_2))^\top$ is majorized by the vector $(\lambda_1(\Delta_1)+\lambda_1(\Delta_2),\ldots,\lambda_n(\Delta_1)+\lambda_n(\Delta_2))^\top$, \eqref{f5.26} implies \eqref{f5.25}. Marcus and Lopez proved by induction on $r$ in \cite[Thm.~1]{ML}  that \eqref{f5.26} is still true if $\frac{E_n,n}{E_{n-1,n}}$ is replaced by $\frac{E_r,n}{E_{r-1,n}}$ for any $r \in \{1,\ldots,n\}$.  To obtain such a result applying the method which was described above, we would have to give a respective extremal representation of any of the functions $\frac{E_r,n}{E_{r-1,n}}$, $r \in \{ 1,\ldots,n-1 \}$. We do not know such a representation. Might be, multilinear algebra cannot be applied here, compare with remarks at the end of the present section.  

By Proposition \ref{p5.1} and Lemma \ref{l2.1} we get a generalization of Minkowski's determinantal inequality. For ${\mathbf t}=(t_1,\ldots,t_n)^\top \in(0,+\infty)^n$ define a function $\psi_{{\mathbf t}}$ by $\psi_{{\mathbf t}}(\Delta)=\prod_{j=1}^n \lambda_{n-j+1}^{t_j}$ for $\Delta \in {\mathcal M}^\geq$, together with a function ${g}_{{\mathbf t}}$ by ${g}_{{\mathbf t}}(x_1,\ldots,x_n) = \prod_{j=1}^n x_j^{t_j}$ for $(x_1,\ldots,x_n)^\top \in {\mathcal D}$. Let $({g}_{{\mathbf t}})\tilde{\,\,}$ be the permutation invariant extension of ${g}_{{\mathbf t}}$ to $[0,+\infty)^n$. 

\begin{pro} \label{p5.11}
    Let ${\mathbf t} \in {\mathcal D}_+$. The functions $\psi_{{\mathbf t}}$ and $({g}_{{\mathbf t}})\tilde{\,\,}$ are concave if and only if $t = \sum_{j=1}^n t_j \leq 1$. They are superadditive if and only if $t \geq 1$. 
\end{pro}

If ${\mathbf t} \in (0,+\infty)^n \setminus {\mathcal D}_+$ the function $({g}_{{\mathbf t}})\tilde{\,\,}$ is known to be not concave, cf.~\cite[3.A.2.b]{MOA}. So it is not Schur concave. We do not know a description of the set of all ${\mathbf t} \in (0,+\infty)^n$ such that functions $\psi_{{\mathbf t}}$ or $({g}_{{\mathbf t}})\tilde{\,\,}$, respectively, are superadditive. In case $n=2$ easy calculations show that these functions can be superadditive for some ${\mathbf t} \not\in {\mathcal D}_+$. 

Let $\phi \in {\mathcal F}$. The definition of $\iota( \cdot ; {\mathcal S}_\phi)$ implies that ${\rm tr}(A \Delta A^*) \geq \iota(\Delta ; {\mathcal A}) = \iota(\Delta ; {\mathcal A})(\phi(A))^2$ for any $A \in {\mathcal S}_\phi$. Since $(\phi(aA))^2 = a^2 (\phi(A))^2$ for $a \in [0,+\infty)$, $A \in {\mathcal M}$, the inequality ${\rm tr}(A \Delta A^*) \geq  \iota(\Delta ; {\mathcal A})(\phi(A))^2$ is true for any $A \in \mathcal M$ and any $\Delta$. Consequently, a lower bound for ${\rm tr}(A \Delta A^*)$ is found for any given $\Delta$, any $A \in \mathcal M$. In particular, if $\Delta = I$ then a lower bound for the Schatten $2$-norm of $A$ can be derived. Another method is based on Lemma \ref{l3.9}(ii). For example, if ${\mathbf t}= (1,\ldots,1)^\top$ and $q= p / (p-2)$, Proposition \ref{p4.2} implies that for given $(x_1,\ldots,x_n)^\top \in (0,+\infty)^n$, the monotonicity of the function $p \in (0,2) \to (\sum_{j=1}^n x_j^p)^{1/p}$ implies the monotonicity of the function $q \in (-\infty,0) \to (\sum_{j=1}^n x_j^q)^{1/q}$. 

Assertions concerning superadditivity of certain functions can also be derived by methods of multilinear algebra. Let $P$ denote the orthogonal projection of the tensor product ${\mathbb C}^n \otimes {\mathbb C}^n$ onto the antisymmetric tensor product ${\mathbb C}^n \wedge {\mathbb C}^n$. For linear operators $A,B \in {\mathbb C}^n$ define $A \wedge B = P (A \otimes B) P$. Recall that ${\mathbb C}^n \wedge {\mathbb C}^n$ is an invariant subspace of $A \wedge A$ and that the eigenvalues of the restriction of $A \otimes A$ to ${\mathbb C}^n \wedge {\mathbb C}^n$ are exactly the numbers $\{ \lambda_i (A)\lambda_k (A) : j,k \in\{ 1,\ldots,n \} \, , \,\, j < k \}$. Since $\Delta_1 \otimes \Delta_2$ is non-negative Hermitian for $\Delta_1,\Delta_2 \in {\mathcal M}^\geq$, the operator $\Delta_1 \wedge \Delta_2$ is non-negative Hermitian as well. 

Thus, ${\rm tr}(\tilde{A} ((\Delta_1+\Delta_2) \wedge (\Delta_1+\Delta_2)) \tilde{A}^* ) \geq {\rm tr}(\tilde{A} (\Delta_1 \wedge \Delta_1) \tilde{A}^*) + {\rm tr}(\tilde{A} (\Delta_2 \wedge \Delta_2) \tilde{A}^*)$ for any $n^2 \times n^2$ matrix $\tilde{A}$ implying the superadditivity of the function $\Delta \in {\mathcal M}^\geq \to \inf \{ {\rm tr}(\tilde{A} (\Delta \wedge \Delta) \tilde{A}^*): \tilde{A} \in {\mathcal A} \}$ for arbitrary subsets $\tilde{{\mathcal A}}$ of the set of $n^2 \times n^2$ matrices. In particular, for the set $\tilde{{\mathcal A}}$ of all $n^2 \times n^2$ matrices satisfying the equality ${\rm tr}(P (A \otimes A) P) = {\rm tr}(A \wedge A)=1$, the function $\frac{E_{n,n}}{E_{n-2,n}}$ is superadditive by Proposition \ref{p2.5}(i) or by Proposition \ref{p4.2}(ii). This result is weaker than the superadditivity of the function  $\Big( \frac{E_{n,n}}{E_{n-2,n}} \Big)^{1/2}$ established by Bullen and Marcus in \cite{BM}.  Our method can be extended to the $k$-th antisymmetric tensor power or even to general symmetric classes, cf.~\cite[Sections 2.3, 2.4]{M}. The consequence of this observation has still to be explored.

\section{Prediction of multivariate stationary sequences}
\label{s6}       

Let $K$ be a complex Hilbert space with inner product $(\cdot,\cdot)$ and corresponding norm $\| \cdot \|$. For a matrix $A=(a_{ik}) \in \mathcal M$ and vectors $X=(x_1,\ldots,x_n)^\top$, $Y=(y_1,\ldots,y_n)$ in $K^n$ define the multiplication by $AX=(\sum_{j=1}^n a_{1j}x_j,\ldots, \sum_{j=1}^n a_{nj} x_j)^\top$ and the Grammian $\mathcal M$-valued inner product $G$ by $G(X,Y)=((x_j,y_k))_{j,k \in \{ 1,\ldots,n \}}$ so that $K^n$ has the structure of a left Hilbert $\mathcal M$-module. Recall that
\begin{equation} \label{f6.1}
    G(AX,BY) = A G(X,Y) B^* \,\, {\rm for} \,\,A,B \in {\mathcal M}, \, X,Y \in K^n  \, .
\end{equation}    
Set $(X,Y)_n = {\rm tr}(G(X,Y)) = \sum_{j=1}^n (x_j,y_j)$ and denote by $\| \cdot \|_n$ the norm defined by the scalar inner product $(\cdot,\cdot)_n$. The notions of a subspace or a submodule of $K$ and $K^n$, respectively, mean always norm-closedness of the selected linear subspaces. For simplicity $\|x-y\|^2$ is called the distance of $x,y \in K$, although the notion `squared distance' would be correct. Similarly, $\| X-Y \|_n^2$ is called the distance of $X,Y \in K^n$. 

Let ${\mathbf X} = (X(s)=(x_1(s),\ldots,x_n(s))^\top )_{s \in {\mathbb Z}}$ be an $n$-variate stationary sequence with values in $K^n$. The subspace $H={\rm c.l.h.}\{ x_j(s) : j=1,\ldots,n , \, s \in {\mathbb Z} \}$ is called the time domain of $\mathbf X$. The submodule $H^n$ is the closed $\mathcal M$-linear hull of all $X(s)$ with $s \in \mathbb Z$, and it is also called the time domain of $\mathbf X$. The probabilistic interpretation of these notions can be seen if $K$ is assumed to be the Hilbert space of square-integrable random variables  with expectation $0$ over a certain probability space. 

For a non-empty subset $\mathcal O$ of ${\mathbb Z} \setminus \{ 0 \}$ let $H({\mathcal O})= {\rm c.l.h.} \{ x_j(s) : j= 1,\ldots,n, \, s \in {\mathcal O} \}$ be the observation domain  based on observations at all points of the observation set $\mathcal O$. Let $P({\mathcal O})=P$ be the orthogonal projection of $H$ onto $H({\mathcal O})$, and let ${\mathbf P}({\mathcal O})={\mathbf P}$ be its $n$-fold inflation, ${\mathbf P}((x_1,\ldots,x_n)^\top) = (P(x_1),\ldots,P(x_n))^\top \in H^n$ for $(x_1, \ldots , x_n)^\top \in H^n$. Recall, that $\mathbf P$ is the $\mathcal M$-linear Grammian orthogonal projection in $H^n$ onto $H({\mathcal O})^n$, i.e. $G(X-{\mathbf P}(X),Y)=0$ for any $X \in H^n$, $Y \in H({\mathcal O})^n$. Therefore,
\begin{equation} \label{f6.2}
    G(X-{\mathbf P}(X),X-{\mathbf P}(X) ) = \min \{ G(X-Y,X-Y) : Y \in H({\mathcal O})^n \}
\end{equation}
with respect to Loewner's semiordering, what implies 
\[
    \| X-{\mathbf P}(X) \|_n^2 = \sum_{j=1}^n \| x_j-P(x_j) \|^2 = \min \left\{ \sum_{j=1}^n \| x_j-y_j \|^2 ; y \in H({\mathcal O}), j = 1,\ldots,n \right\}
\]
for any $X=(x_1,\ldots,x_n)^\top \in H^n$. 

A main goal of multivariate prediction is to determine the prediction error matrix
\begin{equation}  \label{f6.3}
    \Delta({\mathcal O}) = \Delta = G(X(0)-{\mathbf P}(X(0)), X(0)-{\mathbf P}(X(0))) \, .
\end{equation}    
From \eqref{f6.3} we can find that the distance
\begin{equation} \label{f6.4}
    d_j({\mathcal O}) = d_j = \| x_j(0)-P(x_j(0)) \|^2
\end{equation}
of the element $x_j(0)$ to the observation domain is the $j$-th element of the principal diagonal of $\Delta$. 

In practice one would sometimes be interested to predict a certain functional of $X(0)$ instead of $X(0)$ itself, and to compute the corresponding prediction error, cf.~\cite{MMS,V}. The simplest problem of this type is the prediction of a linear combination of the elements of $X(0)$ or, slightly more general, an $\mathcal M$-linear combination of these elements. From \eqref{f6.1}, \eqref{f6.2}, \eqref{f6.3} and the $\mathcal M$-linearity of $\mathbf P$ follows that 
\begin{equation}   \label{f6.5}
\begin{split}
    A \Delta A^*   & =  G(AX(0)-{\mathbf P}(A X(0)), AX(0)-{\mathbf P}(A X(0)) ) =\\
   & \quad \quad \quad =  \min \{ G(A X(0)-Y,A X(0)-Y) : Y \in H({\mathcal O})^n \}\\
\end{split}
\end{equation}
for $A \in \mathcal M$ with respect to Loewner's semiordering. However, recall that for arbitrary subsets ${\mathcal A} \subseteq \mathcal M$ an infimum of the set $\{ A \Delta A^* : A \in {\mathcal A} \}$ might not exist with respect to Loewner's semiordering. Therefore, motivated by the paper \cite{HL} we formulate the following prediction problem of Helson-Lowdenslager type:

\begin{problemHL}
{\rm
For a non-empty subset ${\mathcal A} \in {\mathcal M}$ and a non-empty subset ${\mathcal O} \subseteq {\mathbb Z} \setminus \{ 0 \}$, compute the prediction error
\[
    d({\mathcal O} ; {\mathcal A}) = d = \inf \{ {\rm tr} (G(AX(0)-Y,AX(0)-Y)) : Y \in H({\mathcal O})^n, A \in {\mathcal A} \} \, .
\]
If the infimum is a minimum then describe the set of all $A \in {\mathcal A}$ where the minimum is attained.
}
\end{problemHL}

Since $d({\mathcal O} ; {\mathcal A}) = \inf \{ {\rm tr}(A \Delta({\mathcal O}) A^*): A \in {\mathcal A} \} = \iota ( \Delta({\mathcal O}) ; {\mathcal A})$ by \eqref{f6.5}, one is lead to apply the results of the preceding sections to the problem (HL) and to give them a prediction theoretic meaning. In the sequel we discuss mainly consequences that are of particular interest in multivariate prediction theory. A convenient way is to partition a matrix $A$ into its row vectors
\[
   A = \left( \begin{array}{c} {\mathbf a}^{(1)} \\ \vdots \\ {\mathbf a}^{(n)} \end{array} \right) \in {\mathcal M} \, .
\]
Then ${\rm tr}(A \Delta A^*) = \sum_{j=1}^n \| {\mathbf a}^{(j)} X(0) - P({\mathbf a}^{(j)} X(0) )\|^2$ is the sum of distances of linear combinations of ${\mathbf a}^{(j)}X(0)$, $j\in \{ 0,\ldots,n \}$, to the observation domain, and $d$ is the minimum of these sums if $A$ runs through the set $\mathcal A$. A similar interpretation was given by Helson and Lowdenslager \cite[Thm.~11]{HL}. A more concrete result follows if we apply Proposition \ref{p2.4} with $p=2$.

\begin{pro}   \label{p6.1}
   Let $k \in \{ 1,\ldots, n \}$. The minimal sum of distances of the linear combinations of elements of  $\{ {\mathbf a}_j^\top X(0) : j \in \{ 1,\ldots,k \} \}$, where $\{ {\mathbf a}_1 ,\ldots, {\mathbf a}_k \}$ runs through the set of $k$-tuples of vectors of ${\mathbb C}^n$ with $\sum_{j=1}^k \| {\mathbf a}_j \|^2 =1$, is equal to the smallest eigenvalue of the prediction error matrix $\Delta$.
\end{pro}   

Setting $k=1$ in the preceding proposition we get that the distance of the set $\{ {\mathbf a}^\top X(0) : \| {\mathbf a}\|_2=1 \}$ to the observation domain equals to $\lambda_n$. Since all vectors of the canonical orthonormal basis have Euclidean length $1$ we have $\lambda_n \leq \min \{ d_j : j \in \{ 1,\ldots,n \} \}$ what leads to the well-known fact that the smallest element on the principal diagonal of a non-negative Hermitian matrix is not smaller than the minimal eigenvalue of that matrix. It would be interesting to determine the distance of the set $\{ {\mathbf a}^\top X(0) : {\mathbf a}  \in {\mathcal A}_1 \}$ for other subsets ${\mathcal A}_1 \subseteq \mathcal M$. For instance, we do not know a solution if ${\mathcal A}_1$ is the set of all vectors whose $l_p$-norm equals to $1$ for $p \not= 2$. If in Proposition \ref{p6.1} the condition $\sum_{j=1}^k \| {\mathbf a} \|^2=1$ is replaced by the condition $\sum_{j=1}^k \| {\mathbf a}_j \|^2=k$ the corresponding distance $d$ equals to $k \lambda_n$. It is instructive to compare this fact with Ky Fan's result, which claims that the minimal sum of the distances of $\{ {\mathbf a}_j^\top X(0) : j \in \{ 1,\ldots,k \} \}$ equals to $\sum_{j=1}^k \lambda_{n-j+1}$ if the set $\{ {\mathbf a}_1,\ldots,{\mathbf a}_k \}$ covers all orthonormal systems of $k$ vectors, cf.~\cite[Thm.~20.A.2]{MOA}, \cite{Fan}. 

To apply Proposition \ref{p2.5} and Example \ref{e2.8} note, that in case ${\mathcal A}_1$ is a subset of ${\mathbb C}^n$ and ${\mathcal A}_2$ is a subset of the set of complex-valued $(n-1) \times n$ matrices such that the zero matrix belongs to ${\mathcal A}_2$, then for ${\mathcal A} = \left\{ \left( \begin{array}{c} {\mathbf a}^\top \\ A_2 \end{array} \right) : {\mathbf a} \in {\mathcal A}_1, {\mathcal A}_2 \in {\mathcal A}_2 \right\}$ the prediction error $d({\mathcal O} ; {\mathcal A})$ coincides with the distance of the set $\{ {\mathbf a}^\top X(0) : {\mathbf a} \in {\mathcal A}_1 \}$ to $H({\mathcal O})$. In particular, if ${\mathcal A}_1$ is the singleton $\{ e_1 \}$ then $d=d_1$. Furthermore, if $r \in \{ 1,\ldots,n \}$ and ${\mathcal A}_1$ is the set of all vectors ${\mathbf a} = (a_1 ,\ldots, a_n ) \in {\mathbb C}^n$ with
\begin{equation}   \label{f6.6}
     \sum_{j=1}^r a_j = 1
\end{equation}
then $d$ can be considered as the distance of the set of all linear combinations $\sum_{j=1}^r a_j x_j(0)$ with coefficients satisfying \eqref{f6.6} to the set $H({\mathcal O}) + {\rm l.h.} \{ x_j(0) : j \in \{ r+1,\ldots,n \}\}$. The case $r=1$ is of particular interest. Remembering Rozanov's minimality definition \cite{MMS} (cf.~\cite{KMae} for an overview on minimality notions of multivariate stationary sequences), we call the distances
\begin{equation}   \label{f6.7}
    \tilde{d}_k({\mathcal O}) = \tilde{d}_k = \inf \{ \| x_k(0)-y \|^2 : y \in H({\mathcal O}) + {\rm l.h.} \{ x_j(0) : j \not=k \}\} \, , \,\, k \in \{ 1,\ldots,n\}
\end{equation}
the Rozanov distances. Using $\tilde{d}_k$ we can give a practicable invertibility criterion for the prediction error matrix. 

\begin{pro}   \label{p6.2}
   For any non-empty subset $\mathcal O$ of ${\mathbb Z} \setminus \{ 0 \}$ the following assertions are equivalent:
   \begin{itemize}
      \item[(i)] The prediction error matrix \eqref{f6.3} is invertible.
      \item[(ii)] The elements of $X(0)$ are linearly independent over $H({\mathcal O})$. 
      \item[(iii)] All Rozanov distances \eqref{f6.7} are positive.
   \end{itemize}
\end{pro}

\begin{proof}
(i) $\leftrightarrow$ (ii): Recall that the Grammian matrix $\Delta({\mathcal O})$ is invertible if and only if the elements $\{ x_j(0)-{P}(x_j(0)) : j \in \{ 1,\ldots,n \} \}$ are linearly independent. Additionally, we know that $\sum_{j=1}^n a_j (x_j(0) - {P}(x_j(0))) = 0$ if and only if $\sum_{j=1}^n a_j x_j(0) \in H({\mathcal O})$.

(ii) $\leftrightarrow$ (iii): For any $k \in \{ 1,\ldots,n \}$ the equality $\tilde{d}_k({\mathcal O})=0$ is equivalent to the existence of a representation  $X_k(0) = \sum_{j \not=k} a_jx_j(0) +y$, or equivalently, $x_k(0)-\sum_{j \not=k} a_jx_j(0) = y$ for some coefficients $a_j \in {\mathbb C}$, $y \in H({\mathcal O})$. 
\end{proof}          

Now, we can give Examples {\ref{e2.8}(i),(ii) the prediction theoretical form. Let $r \in \{ 1,\ldots,n \}$ and recall that $\delta_{jk}^{(\dagger)}$ denotes the $(j,k)$-th element of $\Delta^\dagger$. 

\begin{pro} \label{p6.3}
\begin{itemize}
    \item[(i)]   If $e_j \in {\mathcal R}(\Delta)$ for $j \in \{ 1,\ldots,r \}$ then the minimum $\mu_r$ of the set  $\{ \sum_{j=1}^r | a_j|^2 \tilde{d}_j : \sum_{j=1}^r a_j=1, a_j \in {\mathbb C} \}$ is equal to $({\rm tr}((\Delta / \Delta_{22})^{-1}))^{-1}$, where $\Delta_{22}$ denotes the right lower $(n-r) \times (n-r)$ corner of $\Delta$. The minimum $\mu_r$ is attained if and only if $a_j=(\sum_{j=1}^r \delta_{jj}^{(\dagger)} )^{-1}\delta_{jj}^{(\dagger)}$. If $e_j \not\in {\mathcal R}(\Delta)$ for some $j \in \{ 1,\ldots,r \}$ then $\mu_r=0$. 
    \item[(ii)]    If $\sum_{j=1}^r e_j \in {\mathcal R}(\Delta)$ then the distance of the set $\{ \sum_{j=1}^r a_j x_j(0) : \sum_{j=1}^r a_j = 1, a_j \in {\mathbb C} \}$ to the set $H({\mathcal O}) + {\rm l.h.} \{ x_j(0) : j \in \{ (r+1),\ldots,n \} \}$ is equal to $(\sum_{j,k=1}^r \delta_{jk}^{(\dagger)})^{-1}$. If $a_m = (\sum_{j,k =1}^r \delta_{jk}^{(\dagger)} )^{-1} \sum_{k=1}^r \delta_{mr}^{(\dagger)}$, $m \in \{ 1,\ldots,r \}$, then the value is attained, and vice versa. If $\sum_{j=1}^r e_j \not\in {\mathcal R}(\Delta)$ this distance equals zero. 
\end{itemize}    
\end{pro}

Because $\mu_r=0$ if $\tilde{d}_j=0$ for some $j \in \{ 1,\ldots,n \}$, we can derive a generalization of Proposition \ref{p6.2} from Proposition \ref{p6.3}(i): 

\begin{cor}   \label{c6.4}
  For all $r \in \{ 1,\ldots,n \}$, the Rozanov distances $\tilde{d}_j$ are strictly positive for $j \in \{ 1,\ldots,r \}$ if and only if $e_j \in {\mathcal R}(\Delta)$ for $j \in \{ 1,\ldots,r \}$.
\end{cor}   
   
If $\tilde{d}_j > 0$ for any $j \in \{ 1,\ldots,r \}$ then we can show that $\mu_r=(\sum_{j=1}^r \tilde{d}_j^{-1})^{-1}$ leading to another corollary of Proposition \ref{p6.2}(i):

\begin{cor} \label{c6.5}
    If $e_j \in {\mathcal R}(\Delta)$ for any $j \in \{ 1,\ldots,r \}$ then $\sum_{j=1}^r \tilde{d}_j^{-1} = {\rm tr}((\Delta / \Delta_{22})^{-1})$. 
\end{cor}

If $r=1$ one has $\tilde{d}_1=(\delta_{11}^{(\dagger)})^{-1}$ if $e_1 \in {\mathcal R}(\Delta)$, and $\tilde{d}_1=0$ otherwise. Thus, if $\Delta \in {\mathcal I}$ we can say that 
\begin{equation} \label{f6.8}
     \tilde{d}_j = (\delta_{jj}^{(\dagger)})^{-1} = \frac{{\rm det}(\Delta)}{{\rm det}(\Delta_j)} \, , 
\end{equation}
where $\Delta_j$ is obtained from $\Delta$ deleting the $j$-th row and column from the matrix representation of $\Delta$, and $j$ takes any value of $\{ 1,\ldots,n \}$. The interlacing theorem for eigenvalues of Hermitian matrices yields
\[
    \lambda_n \leq \tilde{d}_j \leq \lambda_1 \, .
\]
Moreover, since $\tilde{d}_j$ equals the $j$-th element of the principal diagonal of $\Delta^{-1}$ we get the equality $\sum_{j=1}^n \tilde{d}_j^{-1} = {\rm tr}(\Delta^{-1}) = \sum_{j=1}^n \lambda_j^{-1}$. So, there exist a certain index number $k$ such that $\tilde{d}_k$ is not larger than the harmonic mean of the eigenvalues of $\Delta$. 

The next result describes conditions under which $\tilde{d}_j$ is equal to one of the extremal values $\lambda_1$ or $\lambda_n$:

\begin{pro}   \label{p6.6}
   One of the equalities $\tilde{d}_1=\lambda_1$ or $\tilde{d}_1=\lambda_n$ is satisfied if and only if
   \begin{equation} \label{f6.9}
        (x_1(0)-{P}(x_1(0))) \,\, {\it is} \, {\it orthogonal} \, {\it to} \,\, (x_j(0)-{P}(x_j(0))) \,\, {\it for} \, {\it any} \,\, j \in \{ 2,\ldots,n \}    \,. 
   \end{equation}
   In this case $d_1=\tilde{d}_1$, i.e.~the orthogonal projection of $x_1(0)$ onto $H({\mathcal O})+ {\rm l.h.} \{x_j(0) : j \in \{ 2,\ldots,n \}\}$ coincides with ${P}(x_1(0))$. 
\end{pro}   
        
\begin{proof}
If \eqref{f6.9} is satisfied all assertions are obvious. We prove the converse conclusion. Suppose, $\tilde{d}_1=\lambda_1$ or $\tilde{d}_1=\lambda_n$. Set $\Delta = (\delta_{jk})$, $\delta = (\delta_{21},\ldots,\delta_{n1})^\top$ and let $I_{n-1}$ be the $(n-1) \times (n-1)$ identity matrix. If $\tilde{d}_1=\lambda_1$, \eqref{f6.8} yields $\lambda_j(\Delta)=\lambda_{j+1}(\Delta_1)$ for any $j \in \{ 1,\ldots,n-1 \}$. Thus, if the complex number $z$ is in the resolvent set of $\Delta$ we get the equality
\begin{eqnarray*}
    \prod_{j=1}^n (z-\lambda_j) & = & {\rm det}(zI-\Delta) = {\rm det}((zI_{n-1}-\Delta_1) (z-\delta_{11}- \delta^* (zI_{n-1}-\Delta_1)^{-1} \delta) \\
    & = & \prod_{j=2}^n (z-\lambda_j) (z-\delta_{11}- \delta^* (zI_{n-1}-\Delta_1)^{-1} \delta)
\end{eqnarray*}
or equivalently, $z-\lambda_1=z-\delta_{11}-\delta^*(zI_{n-1}-\Delta_1)^{-1} \delta$, resulting in 
\begin{eqnarray*}
\delta & = & (\left( (x_1(0)-{P}(x_1(0)),  (x_2(0)-{P}(x_2(0))\right), \ldots  \\
& & \quad \quad \ldots , \left((x_1(0)-{P}(x_1(0)),  (x_n(0)-{P}(x_n(0))\right) )^\top \\
& = & 0 \, . 
\end{eqnarray*}
If $\tilde{d}_1=\lambda_n$ the proof is quite similar. 
\end{proof}        

The preceding proposition reveals that $\tilde{d}_1$ is extremal if and only if knowledge of $x_2(0),\ldots,x_n(0)$ does not improve the prediction. Let us mention, however, that the condition \eqref{f6.9} does not mean that the sequence $(x_1(s))_{s \in {\mathbb Z}}$ is orthogonal to any of the sequences $(x_j(s))_{s \in {\mathbb Z}}$, $j\in \{ 2,\ldots,n \}$, in general. 

We finish the present section with another interpretation of $d({\mathcal O};{\mathcal A})$. Let $M$ be the (non-stochastic) spectral measure of the sequence $\mathbf X$. Let ${\mathcal T}({\mathcal O})$ be the set of all $\mathcal M$-valued trigonometric  polynomials of the form $\sum_k A_k e^{{\mathbf i} s_k ( \cdot )}$ for $A_k \in \mathcal M$, $s_k \in \mathcal O$. According to the multivariate version of Kolmogorov's isomorphism theorem we get $d({\mathcal O},{\mathcal A}) = \inf \left\{ {\rm tr} \left( \int_0^{2 \pi} (A-F(\theta)) \, {\mathrm d}(M(\theta))(A-F(\theta))^* \right) : F \in {\mathcal T}({\mathcal O}), A \in {\mathcal A} \right\}$. If $A$ is invertible then $\inf \left\{ {\rm tr} \left( \int_0^{2 \pi} (A-F(\theta)) \, {\mathrm d}(M(\theta))(A-F(\theta))^* \right) : F \in {\mathcal T}({\mathcal O}) \right\} = \inf \left\{ {\rm tr} \left( \int_0^{2 \pi} (I-F(\theta)) \, A\, {\mathrm d}(M(\theta)) \, A^* (I-F(\theta))^* \right) : F \in {\mathcal T}({\mathcal O}), A \in {\mathcal A} \right\}$, and hence, $d({\mathcal O} ; {\mathcal A}) = \inf \{ {\rm tr}(\Delta_A({\mathcal O})) : A \in ({\mathcal A} \cap {\mathcal I} ) \}$ in case $ ({\mathcal A} \cap {\mathcal I})$ is dense in $\mathcal A$ and $\Delta_A({\mathcal O})$ denotes the prediction error matrix of the sequence $(A X(s))_{s \in {\mathbb Z}}$, according to Lemma \ref{l2.2}(ii). Since in applications the modelling stationary sequence is estimated from observations (cf.~\cite[Chapter 16]{Br}) there is some freedom of choice. Therefore, using the preceding results  one could try to choose from a set $\{ A X(s) \}_{s \in {\mathbb Z}}$, where $A$ exhaustes a certain subset $\mathcal A$ of $\mathcal M$ in such a way that $d({\mathcal O},{\mathcal A})$ is as small as possible.


\begin{thebibliography}{00}


\bibitem{A}
 Ando, T.:
 Schur complements and matrix inequalities: operator-theoretic approach.
 in: Fuzhen Zhang (ed.), The Schur Complement and Its Applications, Springer, New York, 2005/2010, 137--162.

\bibitem{Bh}
 Bhatia, R.:
 {\it Matrix Analysis}.
 Graduate Texts in Mathematics v.~169, Springer, New York, 1997.
 
\bibitem{Br}
 Brockwell, P.~J., Davis, R.~A.: 
 {\it Time Series: Theory and Methods}.
 2nd edition, Springer, New York, 2006.
 
\bibitem{BM}
 Bullen, P., Marcus, M.: 
 Symmetric means and matrix inequalities.
 Proc. Amer. Math. Soc. 12 (1961), 285--290. 
 
 \bibitem{Cheng}
 Cheng, R.:
 Some Banach spaces of vector-valued functions and an extremal problem.
 {\it Anal. Math.} {\bf 23} (1997), 205--222.
 
\bibitem{Ding}
 Ding, Chao:
 Variational analysis of the Ky Fan $k$-norm.
 {\it Set-Valued and Variational Analysis} {\bf 25} (2017), 265--296.
 
\bibitem{Fan}
 Fan, Ky:
 On a theorem of Weyl concerning the eigenvalues of linear transformations.
 {\it Proc. Nat. Acad. SCI. USA} {\bf 35} (1949), 652--655.

\bibitem{HL}
 Helson, H.; Lowdenslager, D.:
 Prediction theory and Fourier series in several variables.
 Acta Math. 99 (1958), 165--202.

\bibitem{HJ_alt}
 Horn, R.~A.; Johnson, C.~R.:
 {\it Matrix Analysis}.
 1st edition, Cambridge University Press, Cambridge, 1985.
 
\bibitem{HJ}
 Horn, R.~A.; Johnson, C.~R.:
 {\it Matrix Analysis}.
 2nd edition, Cambridge University Press, Cambridge, 2012.
 
\bibitem{HJ2}
 Horn, R.~A.; Johnson, C.~R.:
 {\it Topics in Matrix Analysis}.
 2nd edition, Cambridge University Press, Cambridge, 2011.
 
\bibitem{K}
 Klotz, L.:
 A matrix generalization of a theorem of Szeg{\H{o}}.
 Analysis Math. 18 (1992), 63--72.

\bibitem{KL}
 Klotz, L.; Lasarow, A.:
 Extremal problems for matrix-valued polynomials on the unit circle and applications
 to multivariate stationary sequences.
 J. Approx. Theory 125 (2003), 42--62.
 
\bibitem{KMae}
 Klotz, L., M\"adler, C.: 
 On the notion of minimality of a $q$-variate stationary sequence.
 Complex Anal. Oper. Theory 12 (2018), no. 1, 1--15. 

\bibitem{KM}
 Klotz, L.; Medina, J.~M.:
 ${\mathcal I}_H$-singularity and ${\mathcal I}_H$-regularity of multivariate stationary processes over LCA groups.
 Probab. and Math. Statistics 41, 1 (2021), 173--192.
 
\bibitem{Liu}
 Liu, Jianzhou:
 Eigenvalue and singular value inequalities of Schur complements.
 in: Fuzhen Zhang (ed.), The Schur Complement and Its Applications, Springer, New York, 2005/2010, 47--82.

\bibitem{MW}
 Makagon, A; Weron, A.:
 $q$-variate minimal stationary processes.
 Studia Math. 59, (1976), 41--52.
 
\bibitem{M}
 Marcus, M.: 
 {\it Finite Dimensional Multilinear Algebra, Part I}. 
 Marcel Dekker Inc., New York, 1973.
 
\bibitem{ML}
 Marcus, M., Lopes, L.: 
 Inequalities for symmetric functions and Hermitian matrices.
 Canad. J. Math. 9 (1957), 305--312. 
 
\bibitem{MM}
 Marcus, M., McGregor, J.~L.:
 Extremal properties of Hermitian matrices.
 Canad. J. Math. 8 (1956), 524--531.

\bibitem{MOA}
 Marshall, A.~W.; Olkin, I.; Arnold, B.~C.:
 {\it Matrix Inequalities: Theory of Majorization and Its Applications}.
 Springer Series in Statistics, New York, 2011 (First edition: Academic Press, New York, 1979).
 
\bibitem{MMS}
 Masyatka, O.; Moklyachuk, M.; Sidej, M.:
 Filtering of multidimensional stationary processes with missing observations.
 Universal J. Math. Appl. {2} (2019), no.~1, 24--32.

\bibitem{R}
 Rozanov, Yu.~A.:
 {\it Stationary Random Processes}.
 Holden-Day, San Francisco, 1967.
 
\bibitem{V}
 Vaninski, K.~L.:
 A method of functional estimation for stationary processes.
 J. Multivariate Anal. {46} (1993), no.~1, 97--111.
 
\bibitem{Watson}
 Watson, G.~A.: 
 On matrix approximation problems with Ky Fan $k$-norms
 {\it Numerical Algorithms} {\bf 5} (1993), 263--272.

\bibitem{WM}
 Wiener, N.; Masani, P.:
 The prediction theory of multivariate stochastic processes, Part I - The regularity condition.
 Acta Math. 98 (1957), 111--150.

\bibitem{Z}
 Zasukhin, V.~N.:
 On the theory of multidimensional stationary random processes (Russian).
 Dokl. Akad. Nauk Acad. SSSR (N.S.) 33 (1941), 435--437.


\end{thebibliography}


\end{document}